\documentclass[11pt,english]{smfart} 
\usepackage{bm}
\usepackage[english,francais]{babel}
\usepackage{amscd}
\usepackage{amssymb,url,xspace,smfthm}
\usepackage{paralist}
\usepackage{datetime} 
\usepackage{colortbl}
\usepackage{color}
\usepackage[dvipsnames]{xcolor}
 \usepackage[%colorlinks,linkcolor=BrickRed,
 hypertexnames=true,linktocpage]{hyperref} 
\usepackage{mathtools}
\usepackage{euscript}
\usepackage[all]{xy}
\usepackage{graphicx} 

\usepackage{nameref}
\usepackage[T1]{fontenc}
\usepackage{newtxtext}
\usepackage{newtxmath}
\usepackage{textcomp}
 \usepackage[text={150mm,251mm},centering, marginparwidth=75pt]{geometry} 
\usepackage{comment}  
\usepackage{cite}  
\usepackage{thmtools}
\usepackage{enumitem}
\usepackage{letltxmacro} 
\usepackage{nameref}
\usepackage{cleveref}
\usepackage{calc}
\usepackage{interval}
\usepackage{manfnt}
\usepackage{tikz-cd}  
\usepackage[utf8]{inputenc}
\usepackage{marginnote}
\usepackage{euscript}

\usepackage{smfhyperref}
\usepackage[hyperpageref]{backref}

\usepackage{faktor}

\theoremstyle{plain}
\newtheorem{thmx}{Theorem}
\renewcommand{\thethmx}{\Alph{thmx}} 
\newtheorem{thm}{Theorem}[section]  
\newtheorem{lem}[thm]{Lemma}
\newtheorem{claim}[thm]{Claim}

\newtheorem{proposition}[thm]{Proposition}
\newtheorem{cor}[thm]{Corollary}

\newtheorem{conjecture}[thm]{Conjecture}
\theoremstyle{definition}
\newtheorem{dfn}[thm]{Definition}

\theoremstyle{remark}
\newtheorem{rem}[thm]{Remark} 
\newtheorem{example}[thm]{Example}


\numberwithin{equation}{subsection}  

\theoremstyle{plain}
\newlist{thmlist}{enumerate}{1}
\setlist[thmlist]{wide = 0pt, labelwidth = 2em, labelsep*=0em, itemindent = 0pt, leftmargin = \dimexpr\labelwidth + \labelsep\relax, noitemsep,topsep = 1ex, font=\normalfont, label=(\roman*), ref=\thethm.(\roman{thmlisti})}

\addtotheorempostheadhook[thm]{\crefalias{thmlisti}{thm}}

\addtotheorempostheadhook[assumpsion]{\crefalias{thmlisti}{assumption}}

\addtotheorempostheadhook[cor]{\crefalias{thmlisti}{cor}}

\addtotheorempostheadhook[proposition]{\crefalias{thmlisti}{proposition}}

\addtotheorempostheadhook[dfn]{\crefalias{thmlisti}{dfn}}

\addtotheorempostheadhook[lem]{\crefalias{thmlisti}{lem}}
\addtotheorempostheadhook[main]{\crefalias{thmlisti}{main}}

\addtotheorempostheadhook[rem]{\crefalias{thmlisti}{rem}}

\newlist{thmenum}{enumerate}{1} % also creates a counter called 'propenumi'
\setlist[thmenum]{wide = 0pt, labelwidth = 2em, labelsep*=0em, itemindent = 0pt, leftmargin = \dimexpr\labelwidth + \labelsep\relax, noitemsep,topsep = 1ex, font=\normalfont, label=(\roman*), ref=\thethmx.(\roman{thmenumi})}%{label=\alph*), ref=\thethmx~(\alph*)}
\crefalias{thmenumi}{thmx} 

\newlist{corlist}{enumerate}{1} % also creates a counter called 'propenumi'
\setlist[corlist]{wide = 0pt, labelwidth = 2em, labelsep*=0em, itemindent = 0pt, leftmargin = \dimexpr\labelwidth + \labelsep\relax, noitemsep,topsep = 1ex, font=\normalfont, label=(\roman*), ref=\thecorx.(\roman{corlisti})}%{label=\alph*), ref=\thethmx~(\alph*)}
\crefalias{corlisti}{corx} 

\crefname{lem}{Lemma}{Lemmas} 
\crefname{conjecture}{Conjecture}{Conjectures}
\crefname{thm}{Theorem}{Theorems}
\crefname{proposition}{Proposition}{Propositions}
\crefname{dfn}{Definition}{Definitions}
\crefname{rem}{Remark}{Remarks}
\crefname{cor}{Corollary}{Corollaries}
\crefname{corx}{Corollary}{Corollaries}
\crefname{problem}{Problem}{Problems}
\crefname{thmx}{Theorem}{Theorems}
\crefname{claim}{Claim}{Claims}
\crefname{assumption}{Assumption}{Assumptions}
\crefname{main}{Main Theorem}{Main Theorems}

\makeatletter
\newcommand*{\rom}[1]{\expandafter\@slowromancap\romannumeral #1@}
\makeatother

\makeatletter     %Replace the Section ...  by the symbol \S ...  in Cref
\newcommand{\crefnames}[3]{%
	\@for\next:=#1\do{%
		\expandafter\crefname\expandafter{\next}{#2}{#3}%
	}%
}
\makeatother

\crefnames{part,chapter,section}{\S}{\S\S}

\newcommand{\cA}{\mathcal A}

\newcommand{\cO}{\mathcal O}

\newcommand{\bA}{\mathbb{A}}

\newcommand{\bC}{\mathbb{C}}
\newcommand{\bD}{\mathbb{D}}

\newcommand{\bQ}{\mathbb{Q}}

\newcommand{\bZ}{\mathbb{Z}}

\newcommand{\Spalg}{\mathrm{Sp}_{\mathrm{alg}}}

 \title[Fundamental groups of special varieties]{Hyperbolicity and fundamental groups of complex quasi-projective varieties (III): applications} 

\date{\today} 

\author[B. Cadorel]{Beno\^{i}t Cadorel} 
\email{benoit.cadorel@univ-lorraine.fr}
\address{Institut \'Elie Cartan de Lorraine, Universit\'e de Lorraine, F-54000 Nancy,
	France}
\urladdr{http://www.normalesup.org/~bcadorel/} 

\author[Y. Deng]{Ya Deng}

\email{ya.deng@math.cnrs.fr, deng@imj-prg.fr}
\address{CNRS,  
	Institut de Math\'ematiques de Jussieu-Paris Rive Gauche,
	Sorbonne Universit\'e, Campus Pierre et Marie Curie,
	4 place Jussieu, 75252 Paris Cedex 05, France}
\urladdr{https://ydeng.perso.math.cnrs.fr}

\author[K. Yamanoi]{Katsutoshi Yamanoi}

\email{yamanoi@math.sci.osaka-u.ac.jp}
\address{Department of Mathematics, Graduate School of Science, Osaka University, Toyonaka,  Osaka 560-0043, Japan} 
\urladdr{https://sites.google.com/site/yamanoimath/}

\begin{document}

	\begin{abstract} 
This paper is Part~III of a series of three. We begin by introducing the notion of \emph{$h$-special varieties}, which can be seen as varieties “chain-connected by the Zariski closures of entire curves.” We prove that if $X$ is either a special smooth complex quasi-projective  variety in the sense of Campana or an $h$-special variety, then for any linear representation $\varrho:\pi_1(X)\to \mathrm{GL}_N(\mathbb{C})$, the image $\varrho(\pi_1(X))$ is virtually nilpotent. We also provide examples showing that this result is sharp, leading to a revised form of Campana’s abelianity conjecture for smooth quasi-projective varieties. In addition, we prove a structure theorem for quasi-projective varieties  with big and semisimple representations of the fundamental groups, thereby addressing a conjecture   by Kollár in 1995. We also construct several examples of quasi-projective varieties that are special and $h$-special, highlighting certain atypical properties of the non-compact case in contrast with the projective setting.
	\end{abstract}  
 
\subjclass{32Q45,    14D07,  20F65}
\keywords{special varieties,  nilpotent group}

%\altkeywords{pseudo Picard hyperbolicité, conjecture de Green-Griffiths-Lang généralisée, variétés spéciales, applications harmoniques vers des immeubles de Bruhat-Tits, variations de structures de Hodge, théorie de Nevanlinna, conjugués de Galois}

\maketitle
  \tableofcontents

	 \section{Introduction} 
 \subsection{Fundamental groups of  special and $h$-special varieties}
In  Campana's birational classification of algebraic varieties \cite{Cam04,Cam11}, he introduced the notion of special varieties as the geometric opposite of varieties of  general type. Roughly speaking, a complex quasi-projective variety is special if none of its birational models admits a dominant morphism onto a positive-dimensional quasi-projective variety whose orbifold base is of  log general type.
 Typical examples of special varieties include rationally connected varieties, abelian varieties, and Calabi–Yau varieties.   We refer the readers to \cref{def:special} for the precise definition of \emph{special}   varieties. %Thus, the class of special varieties brings together objects of very different birational nature, but which share the common feature of being “far from general type.”

%Campana’s philosophy is that every variety admits a canonical fibration, the core, whose base is of general type and whose fibers are special. In this way, special varieties play the role of the “building blocks complementary to general type” in the classification theory of algebraic varieties.

An algebraic variety is called \emph{Brody special} if it contains a Zariski dense entire curve. Campana conjectured \cite{Cam04,Cam11} that a variety is special if and only if it is Brody special. In this paper, we introduce a (weaker) notion of \emph{$h$-special varieties}, which, roughly speaking, can be viewed as varieties that are “chain-connected by Brody special varieties”.  For the definition of $h$-special above, we refer the readers to \cref{defn:20230407}; it includes smooth quasi-projective varieties $X$ such that
\begin{itemize}
	\item $X$ admits Zariski dense entire curves, or 
	\item two general points of $X$ are connected by the chain of entire curves.
\end{itemize}   

In recent years, Campana’s special varieties have been the subject of intensive study from both geometric and arithmetic perspectives (see \cite{BCJW,BJL24,BJR,JR22,BJR26}, among others). 
In this paper, we focus primarily on the properties of the fundamental groups of special and $h$-special varieties. A fundamental conjecture of Campana \cite[Conjecture 13.10.(1)]{Cam11} predicts that a smooth quasi-projective variety $X$ that is either special or Brody special has a virtually abelian fundamental group. When $X$ is projective and special, it is known that all linear quotients of $\pi_1(X)$ are virtually abelian (cf. \cite[Theorem 7.8]{Cam04}). Moreover, in \cite[Theorem 1.1]{Yam10}, the third author established the same conclusion for Brody special smooth projective varieties. It is therefore natural to expect the same  results in the quasi-projective case.

The first result of this paper is the construction of a quasi-projective surface that is both special and Brody special,   and whose fundamental group is linear and nilpotent but not virtually abelian (cf. \Cref{example}). This provides a counterexample to Campana’s conjecture in full generality. Motivated by this, we propose the following revised form of Campana’s conjecture, which incorporates our notion of $h$-special varieties.
\begin{conj}\label{conj:revised2}
	Let $X$ be a   special  or $h$-special smooth quasi-projective variety.     Then  $ \pi_1(X) $ is \emph{virtually nilpotent}. 
\end{conj} 
 Our  first main result in this paper  is  as follows.  
\begin{thmx}[=\cref{thm:VN}]\label{main5}
	Let $X$ be a smooth quasi-projective variety which is either  {special} or  {$h$-special}. Let $\varrho: \pi_1(X) \to \mathrm{GL}_N(\mathbb{C})$ be a linear representation of its fundamental group. 
	Then   its image $\varrho(\pi_1(X))$ is  {virtually nilpotent}.
\end{thmx}

Note that the above theorem is sharp, as shown by \Cref{example}.  Its proof is based on   Nevanlinna theory, Deligne's mixed Hodge theory and hyperbolicity results proven in  \cite[Theorem A]{CDY22}.  
%\begin{thmx}[=\cref{thm:202210123}]\label{main:geomety group}
%Let $X$ be an $h$-special or special  smooth quasi-projective variety.
%Let $G$ be a connected, solvable algebraic group defined over $\mathbb C$.
%Assume that there exists a Zariski dense representation $\varphi:\pi_1(X)\to G(\bC)$.
%Then $G$ is nilpotent.
%In particular, $\varphi(\pi_1(X))$ is nilpotent.
%\end{thmx} 

\subsection{A structure theorem}
In \cite[Conjecture 4.18]{Kol95}, Kollár raised the following conjecture on the structure of varieties with big fundamental group.
\begin{conj}[Koll\'ar]\label{conj:Kollar}
	 Let $X$ be a smooth projective variety with big fundamental group such that $0<\kappa(X)<\operatorname{dim} X$. Then $X$ has a finite étale cover $p: X^{\prime} \rightarrow X$ such that $X^{\prime}$ is birational to a smooth family of abelian varieties over a projective variety of general type $Z$ which has big fundamental group. 
\end{conj}
 Our second main result  is  the following structure theorem, inspired by  \Cref{conj:Kollar}.
\begin{thmx}[=\cref{thm:structure,thm:char}]
\label{thm:20230510}
	Let $X$ be a   smooth quasi-projective variety and let $\varrho:\pi_1(X)\to {\rm GL}_N(\bC)$ be a reductive and big representation. Then
	\begin{thmenum}
	\item the logarithmic Kodaira dimension $\bar{\kappa}(X)\geq 0$. Moreover, if $\bar{\kappa}(X)=0$, then   $\pi_1(X)$ is virtually abelian.
	\item There is a proper Zariski closed subset $Z$ of $X$ such that each non-constant morphism $\bA^1\to X$ has image in $Z$.
	\item After replacing $X$ by a finite \'etale cover and a birational modification, there are a semiabelian variety $A$, a smooth quasi-projective variety $V$, and a birational morphism $a:X\to V$ such that we have the following commutative diagram
		\begin{equation*}
			\begin{tikzcd}
				X \arrow[rr,    "a"] \arrow[dr,  "j"] & & V \arrow[ld, "h"]\\
				& J(X)&
			\end{tikzcd}
		\end{equation*}
	where $j$ is the logarithmic Iitaka fibration and $h:V\to J(X)$ is a  locally analytically trivial fibration with fibers isomorphic to $A$.   Moreover, for a   general fiber $F$ of $j$, $a|_{F}:F\to  A$ is proper in codimension one. 
		\item \label{char abelian} If $X$ is   special or $h$-special, then $\pi_1(X)$ is virtually abelian, and after replacing $X$ by a finite \'etale cover, its quasi-Albanese morphism $\alpha:X\to \cA$ is birational.  
 		\end{thmenum} 
\end{thmx} 
When $X$ is projective, the third author proved \cref{char abelian} in \cite{Yam10} without assuming that the representation $\varrho$ is reductive. However, it is worth noting that when $X$ is only quasi-projective,   \cref{char abelian} might fail if $\varrho$ is not reductive (cf. \cref{rem:sharp abelian}).

This paper constitutes the final part of the long preprint \cite{CDY22original} on arXiv, which has been divided into three parts for journal submission. It corresponds to Sections 10–12 of \cite{CDY22original}. The focus here is mainly on applications of the hyperbolicity results for quasi-projective varieties established in \cite{CDY22}.

Finally, let us mention that, building on the main results of this paper, in \cite{CDHP26} Cao, Hacon, P\u{a}un, and the second author recently strengthened the assertion of \cref{main5} concerning special varieties. More precisely, they prove that, for a special normal quasi-projective variety, the image of any linear representation of its fundamental group is virtually nilpotent of class at most two \cite{CDHP26}. To prove this stronger statement, they develop a deformation theory for suitable local systems on quasi-projective varieties and combine this theory with \cref{main5,pro:202210131} and the main idea from the proof of \cref{thm:VN}. The case of $h$-special varieties is not addressed in \cite{CDHP26}. 
\subsection*{Convention and notation.}In this paper, we use the following conventions and notations:
\begin{itemize}[noitemsep]
	\item Quasi-projective varieties and their closed subvarieties are usually assumed to be irreducible unless otherwise stated, while Zariski closed subsets might be reducible.
	\item Fundamental groups are always referred to as topological fundamental groups.
	\item If $X$ is a complex space, its normalization is denoted by $X^{\mathrm{norm}}$.
%	\item All algebraic groups are assumed to be linear.
		%\item All non-Archimedean local fields are assumed to have  characteristic zero.
	\item An \emph{algebraic fiber space} is a dominant morphism between quasi-projective normal varieties whose general fibers are connected, but not necessarily proper or surjective.
	\item A birational map $f:X\to Y$ between quasi-projective normal varieties is \emph{proper in codimension one} if there is an open set $Y^\circ\subset Y$ with $Y\backslash Y^\circ$ of codimension at least two such that $f$ is proper over $Y^\circ$.
	\item All birational modifications throughout the paper are \emph{proper}.
 \item For a normal quasi-projective variety $X$, its logarithmic Kodaira
	dimension is defined to be $\bar{\kappa}(X):=\bar{\kappa}(X')$, where
	$X'\to X$ is any resolution of singularities. Since $\bar{\kappa}$ is
	invariant under proper birational morphisms of smooth varieties, this is independent of the choice of resolution. 	
	\item Throughout this paper, unless otherwise mentioned, all semisimple algebraic groups are
	assumed to be \emph{positive-dimensional}.
\item For a Weil divisor $D = \sum_i a_i D_i$, with the $D_i$ distinct prime divisors and $a_i \neq 0$, we write $\operatorname{Supp}(D) = \bigcup_i D_i$ for the \emph{support} of $D$, and $D_{\rm red} := \sum_i D_i$ for the corresponding reduced divisor. 
\item For a divisor $D$ on a  semiabelian variety $A$,  let  \(\mathrm{St}(D)\) be the  stabilizer of the reduced divisor $D_{\rm red}$ under the action of \(A\).  
\end{itemize}
 \subsection*{Acknowledgment.}  
 We would like to thank Michel Brion, Patrick Brosnan, Yohan Brunebarbe,  Fr\'ed\'eric Campana, Philippe Eyssidieux, Hisashi Kasuya, Chikako Mese, Gopal Prasad, Guy Rousseau, and Carlos Simpson for many helpful discussions.  
We are especially grateful to Ariyan Javanpeykar for his valuable comments on the paper.  
B.C. and Y.D. acknowledge support from the ANR grant Karmapolis (ANR-21-CE40-0010).  
K.Y. acknowledges support from JSPS Grant-in-Aid for Scientific Research (C) 22K03286. Lastly, we sincerely thank the anonymous referee for his exceptional careful reading and for valuable comments and suggestions that greatly improved the paper.

\section{Technical preliminaries}\label{sec:pre}
We begin by recalling some technical preliminaries, for which details can be found in \cite{NW13,CDY22,CDY25}.
\subsection{Semi-abelian varieties and quasi-Albanese morphisms.} In this subsection we collect some results on  semi-abelian varieties.  

 A commutative complex algebraic  group $\cA$ is called a \emph{semi-abelian variety} if there is a short exact sequence of complex algebraic groups  
$$
1 \rightarrow H \rightarrow \cA \rightarrow \cA_0 \rightarrow 1
$$
where $\cA_0$ is an abelian variety and $H \cong (\bC^*)^\ell$.

We will need the following results in \cite[Propositions 5.6.21 \& 5.6.22]{NW13}).  
\begin{proposition} \label{prop:Koddimabb}
	Let \(\cA\) be a semi-abelian variety. 
	\begin{enumerate}[label=(\alph*)]
		\item Let \(X \subset \cA\) be a closed subvariety. Then \(\overline{\kappa}(X) \geq 0\) with equality if and only if \(X\) is a translate of a semi-abelian subvariety.
		\item Let \(Z \subset \cA\) be a Zariski closed subset. Then \(\overline{\kappa}(\cA - Z) \geq 0\) with equality if and only if \(Z\) has no component of codimension \(1\).  \qed
	\end{enumerate}
\end{proposition}
 
 \subsection{Some factorisation results}\label{sec:fac}
Throughout this paper an \emph{algebraic fiber space} $f:X\to Y$ is a dominant (not necessarily proper) morphism $f$ between quasi-projective normal varieties $X$ and $Y$ such that   general fibers of $f$ are connected.  
\begin{lem}[\protecting{\cite[Lemma 2.1]{CDY22}}]\label{lem:Stein}
	Let $f:X\to Y$ be a morphism between smooth quasi-projective varieties. Then $f$ factors through   morphisms $\alpha:X\to S$ and $\beta:S\to Y$ such that
	\begin{enumerate}[label={\rm (\alph*)}]
		\item $S$ is a quasi-projective normal variety;
		\item   $\alpha$ is an algebraic fiber space;
		\item $\beta$ is a finite morphism.
	\end{enumerate}
	Such a factorisation is unique. \qed
\end{lem} 
The previous factorisation will be called \emph{quasi-Stein factorisation} in this paper.  
\begin{proposition}[\protecting{\cite[Proposition 2.5]{CDY22}}]\label{lem:kollar}
	Let $X$ be a quasi-projective normal variety.  Let $\varrho:\pi_1(X)\to G(K)$ be a representation, where $G$ is a linear algebraic group defined over a field $K$ of zero characteristic.  Then there is a diagram
	\[
	\begin{tikzcd}
		\widetilde{X} \arrow[r, "\mu"] \arrow[d, "f"] & \widehat{X} \arrow[r, "\nu"] & X\\
		Y                                             &  &
	\end{tikzcd}
	\]
	where \(Y\) and \(\widetilde{X}\) are smooth quasi-projective varieties, and
	\begin{enumerate}[label=(\alph*)]
		\item $\nu:\widehat{X}\to X$ is a finite étale cover;
		\item \(\mu : \widetilde{X} \to \widehat{X}\) is a birational proper morphism;
		\item $f : \widetilde{X} \to Y$ is a dominant morphism with connected general fibers;
	\end{enumerate}
	such that there exists a big representation \(\tau : \pi_{1}(Y) \to G(K)\) with  $f^*\tau=(\nu\circ \mu)^*\varrho$.   \qed \end{proposition}

\subsection{Varieties with zero logarithmic Kodaira dimension} \label{sec:quasi}
 \begin{lem}[\protecting{\cite[Lemma 1.2]{CDY25}}]  \label{lem:same Kd}
Let $f:X\to Y$ be a dominant birational morphism of  smooth quasi-projective varieties. 
Let $E\subset X$ be a Zariski closed subset such that $\overline{f(E)}$ has codimension at least two.  
Assume $\bar{\kappa}(Y)\geq 0$.
Then $\bar{\kappa}(X-E)=\bar{\kappa}(X)$.  \qed
\end{lem} 

\begin{lem}[\protecting{\cite[Lemmas 1.3 \& 1.4]{CDY25}}]
Let $\alpha:X\to \cA$ be  a  (possibly non-proper)   morphism  from a smooth quasi-projective variety $X$ to a semi-abelian variety $\cA$ with   $\overline{\kappa}(X)=0$.   
\begin{thmlist}
	\item \label{lem:abelian pi0} if $\alpha$ is birational, there exists a Zariski closed subset $Z\subset \cA$ of codimension at least two such  that $\alpha$ is   isomorphic
	over $\cA\backslash Z$.
	\item  \label{lem:abelian pi} if $\dim X=\dim \alpha(X)$,  $ \pi_1(X)$ is   abelian.    \qed
\end{thmlist}   
\end{lem}
  \subsection{Hyperbolicity of varieties with big fundamental groups}
  In \cite{CDY22}, we prove  hyperbolicity of quasi-projective varieties admitting certain local systems. 
	 \begin{thm}[\protecting{\cite[Theorem A]{CDY22}}]\label{main2}
	Let $X$ be a complex   quasi-projective normal variety and let $G$ be a semisimple algebraic group over $\bC$. If  $\varrho:\pi_1(X)\to G(\bC)$ is a big  and Zariski dense representation, then   there is a proper Zariski closed subset $Z\subsetneqq X$  such that
	\begin{thmlist}
		\item \label{main:log general type}    any closed  subvariety   of $X $ not contained in $Z$ is  of log general type. In particular, $X $ is of log general type.
		 	\item  The variety $X $ is \emph{Picard hyperbolic} modulo $Z$, i.e.,  for any holomorphic map $f:\bD^*\to X$ such that $f(\bD^*)\not\subset Z$, it extends to a holomorphic map $\bar{f}:\bD\to \bar{X}$, where $\bar{X}$ is some (thus any) projective compactification of $X$. 
		\item \label{main:pseudo Picard}   The variety $X $ is \emph{pseudo Brody hyperbolic}, i.e.,  all (non-constant) entire curves in $X $ lie on $Z$.  \qed
	\end{thmlist} 
\end{thm}
Recall that $\varrho:\pi_1(X)\to G(\bC)$  is big, if for any positive-dimensional closed subvariety $Z$ of $X$ passing to a very general point of $X$, $\varrho({\rm Im}[\pi_1(Z^{\rm norm})\to \pi_1(X)])$ is an infinite group.

  \section{Special varieties and $h$-special varieties} \label{sec:spechspecial}
Special varieties are introduced by Campana \cite{Cam04,Cam11} in his  remarkable program of classification of geometric orbifolds.  In this subsection we briefly recall definitions and properties of special varieties, and we refer the readers to \cite{Cam11} for more details. 
 
Let $f : X \to Y$ be a dominant morphism between quasi-projective smooth varieties with connected general fibers, that admits a compactification \(\overline{f} : \overline{X} \to \overline{Y}\), where \(\overline{X} = X \sqcup D\) (resp. \(\overline{Y} = Y \sqcup G\)) is a compactification with simple normal crossing boundary divisor.  We will consider $(\overline{X}| D)$ as a geometric orbifold  defined in \cite[D\'efinition 2.1]{Cam11} and $\bar{f}:(\overline{X}| D)\to \overline{Y}$ as  an orbifold morphism  defined in \cite[D\'efinition 2.4]{Cam11} .  In \cite[\S 2.1]{Cam11}, Campana defined the \emph{multiplicity divisor} $\Delta(\bar{f}, D)\subset \overline{Y}$ of $\overline{f}$, for which the \emph{orbifold base} of $\overline{f}$ is the pair $(\overline{Y}|\Delta(\bar{f}, D))$. Note that with this definition, one has \(\Delta(\bar{f}, D) \geq G\). 

 The \emph{Kodaira dimension} of $\overline{f}:(\overline{X}| D)\to \overline{Y}$, denoted by $\kappa(\overline{f}, D)$, is defined to be 
\begin{align} \label{eq:Kodaira}
\kappa(\overline{f}, D):=\inf \{\kappa(\overline{Y}'|\Delta(\bar{f}', D')) \},
\end{align} 
where $\overline{f}':\overline{X}'\to \overline{Y}'$ ranges over all \emph{birational models} of $\overline{f}$; i.e., $\overline{f}':\overline{X}'\to \overline{Y}'$ is an algebraic fiber space between smooth projective varieties   such that we have the following commutative diagram
\begin{equation*}
	\begin{tikzcd}
		\overline{X}' \arrow[r, "u"]\arrow[d, "\overline{f}'"] & \overline{X}\arrow[d, "\overline{f}"] \\
		\overline{Y}' \arrow[r, "v"] & \overline{Y} 
 	\end{tikzcd}
\end{equation*}
where $u$ and $v$ are  birational morphisms, and \(D' = u^{-1}(D)\) is a simple normal crossing divisor on \(\overline{X}'\). We note that $\kappa(\bar{f},D)$ is thus a birational invariant and does not depend on the choice of  the compactification  of $X$. Hence we  define the Kodaira dimension of $f:X\to Y$ to be
\begin{align}\label{eq:Kodaira2}
	 \kappa(Y,f):=\kappa(\bar{f},D).
\end{align} 

 \begin{dfn}[Campana's specialness]\label{def:special}
Let $X$ be a  smooth quasi-projective variety. 
	 We say that $X$ is \emph{weakly special} if for any finite \'etale cover $\widehat{X}\to X$ and any proper birational modification $\widehat{X}'\to \widehat{X}$, there exists no dominant morphism  $\widehat{X}'\to Y$  with  connected general fibers such that $Y$ is a positive-dimensional quasi-projective normal variety of log general type. The variety $X$ is \emph{special} if for  any proper birational modification $\widehat{X}\to X$ and any  dominant morphism $f:\widehat{X}\to Y$ over a positive-dimensional quasi-projective normal  variety $Y$ with connected general fibers,  we have $\kappa(Y,f)<\dim  {Y}$.	
	 
	 A (possibly singular) quasi-projective variety $X$ is called   {weakly special},
	 (resp. {special})  if some (equivalently, any) resolution of singularities of $X$ is
	 weakly special (resp. special).
	 \end{dfn}
	 
	 \begin{rem}\label{rem:singular weakly special}
As pointed us by the referee,  weak specialness of a singular variety $X$ is
	 tested on a desingularization $\widetilde{X}\to X$, and in particular one allows all finite
	 \'etale covers of $\widetilde{X}$, not only those induced by finite \'etale covers of $X$.
	 This is the convention of \cite[Definition 2.10]{BJL24}. It
	 should be emphasised that requiring instead  for
	 the finite \'etale covers of $X$ itself would give an a priori different notion, since a
	 finite \'etale cover of $\widetilde{X}$ need not descend to $X$; see \emph{loc.\ cit.} for an
	 example. All the main results of this paper concern smooth varieties, for which the two
	 conventions coincide.
	 \end{rem}

Campana defined $X$ to be \emph{$H$-special} if $X$ has vanishing Kobayashi pseudo-distance.
Motivated by \cite[Definition 9.1]{Cam04}, we introduce the following definition.

\begin{dfn}[$h$-special]\label{defn:20230407}
Let $X$ be a  quasi-projective variety.
We define the equivalence relation $x\sim y$ of two points $x,y\in X$ if and only if there exists a sequence of holomorphic maps $f_1,\ldots,f_l:\mathbb C\to X$ such that letting $Z_i\subset X$ to be the Zariski closure of $f_i(\mathbb C)$, we have 
$$x\in Z_1, Z_1\cap Z_2\not=\emptyset, \ldots, Z_{l-1}\cap Z_l\not=\emptyset, y\in Z_l.$$
We set $R=\{ (x,y)\in X\times X; x\sim y\}$.
We define $X$ to be \emph{hyperbolically special} ($h$-special for short) if and only if  $R\subset X\times X$ is Zariski dense.   
\end{dfn}

By definition, rationally connected projective varieties are $h$-special without referring to a theorem of Campana and Winkelmann \cite{CW16}, who proved that all rationally connected projective varieties contain Zariski dense entire curves.  

\begin{lem}\label{lem:20230406}
If a smooth quasi-projective variety $X$ admits a Zariski dense entire curve $f:\mathbb C\to X$, then $X$ is $h$-special.
\end{lem}

\begin{proof}
Since $f:\mathbb C\to X$ is Zariski dense, we have $x\sim y$ for all $x,y\in X$.
Hence $R=X\times X$.
Thus $X$ is $h$-special.
\end{proof}

Note that the converse of \cref{lem:20230406} does not hold in general (cf. \Cref{ex:20230418}).
For the smooth case, motivated by Campana's suggestion \cite[11.3 (5)]{Cam11b}, we may ask whether a smooth quasi-projective variety $X$ is $h$-special  if and only if it is special. 
Campana proposed the following tantalizing \emph{abelianity conjecture} (cf. \cite[11.2]{Cam11b}).
 \begin{conjecture}[Campana]\label{conj:Campana}
 	A special smooth projective geometric orbifold   (e.g. smooth quasi-projective variety)  has \emph{virtually abelian} fundamental group.  
 \end{conjecture}
 In \Cref{example}  below we give an example of a special (and $h$-special) smooth quasi-projective variety with \emph{nilpotent} but not virtually abelian fundamental group, which thus disproves \cref{conj:Campana} for non-compact smooth quasi-projective varieties. Therefore, we revise Campana's conjecture as follows. 
 \begin{conjecture}[Nilpotency conjecture]\label{conj:revised}
 	An $h$-special or   special smooth quasi-projective variety has \emph{virtually nilpotent} fundamental group.  
 \end{conjecture}

\section{Some properties of $h$-special varieties}\label{sec:20230406}
One can easily see that if $X$ is weakly special, then any finite \'etale cover $X'$ of $X$ is also weakly special.  
It is proven by Campana in \cite{Cam11} that this result also holds for special varieties. 
\begin{thm}[{\cite[Proposition 10.11]{Cam11}}]\label{special} 
If $X$ is  special smooth quasi-projective variety, then any finite \'etale cover $X'$ of $X$ is  special. Moreover, $X$ is weakly special.  \qed
\end{thm} 

For $h$-special varieties, we have the following.

\begin{lem}\label{lem:202304061}
Let $X$ be an $h$-special quasi-projective variety, and let $p:X'\to X$ be a finite \'etale morphism from a quasi-projective variety $X'$.
Then $X'$ is $h$-special.
\end{lem}

\begin{proof}
Let $\sim'$ be the equivalence relation on $X'$ and $R'\subset X'\times X'$ be the set defined by $R'=\{(x',y')\in X'\times X'; x'\sim'y'\}$.
Let $q:X'\times X'\to X\times X$ be the induced map.

We shall show $R\subset q(R')$.
So let $(x,y)\in R$.
Then there exists a sequence $f_1,\ldots,f_l:\mathbb C\to X$ such that 
$$x\in Z_1, Z_1\cap Z_2\not=\emptyset,\ldots,Z_{l-1}\cap Z_l\not=\emptyset, y\in Z_l,$$
where $Z_i\subset X$ is the Zariski closure of $f_i(\mathbb C)\subset X$.
Let $Z_i^{\mathrm{norm}}\to Z_i$ be the normalization.
Then we may take a connected component $W_i$ of $Z_i^{\mathrm{norm}}\times_XX'$ such that 
\begin{equation}\label{eqn:202304061}
\varphi_1(W_1)\cap \varphi_{2}(W_{2})\not=\emptyset,\ldots, \varphi_{l-1}(W_{l-1})\cap \varphi_{l}(W_{l})\not=\emptyset,
\end{equation}
where $\varphi_i:W_i\to X'$ is the natural map.
Note that the induced map $W_i\to Z_i^{\mathrm{norm}}$ is \'etale.
Hence $W_i$ is normal.
Since $W_i$ is connected, $W_i$ is irreducible (cf. \cite[\href{https://stacks.math.columbia.edu/tag/0357}{Tag 0357}]{stacks-project}). 
Since $f_i:\mathbb C\to Z_i$ is Zariski dense, we have a lift $f_i':\mathbb C\to Z_i^{\mathrm{norm}}$.
Since $W_i\to Z_i^{\mathrm{norm}}$ is finite \'etale, we may take a lift $g_i:\mathbb C\to W_i$, which is Zariski dense.
Then $\varphi_i\circ g_i:\mathbb C\to X'$ has Zariski dense image in $\varphi_i(W_i)\subset X'$.

We take $a\in W_1$ and $b\in W_l$ such that $p\circ\varphi_1(a)=x$ and $p\circ\varphi_l(b)=y$.
Then $q((\varphi_1(a),\varphi_l(b)))=(x,y)$.
We have $\varphi_1(a)\in\varphi_1(W_1)$ and $\varphi_l(b)\in\varphi_l(W_l)$.
Hence by \eqref{eqn:202304061}, we have $(\varphi_1(a),\varphi_l(b))\in R'$.
Thus $R\subset q(R')$.

Now let $\overline{R'}\subset X'\times X'$ be the Zariski closure.
To show $\overline{R'}= X'\times X'$, we assume contrary that $\overline{R'}\not= X'\times X'$.
Then $\mathrm{dim}\overline{R'}<2\dim X'$.
This contradicts to $X\times X=\overline{R}\subset q(\overline{R'})$.
Hence $\overline{R'}= X'\times X'$.
\end{proof}

\begin{lem}\label{lem:202304063}
Let $X$ be an $h$-special quasi-projective variety.
Let $S$ be a quasi-projective variety and let $p:X\to S$ be a dominant morphism.
Then $S$ is $h$-special.
\end{lem}

\begin{proof}
Let $\sim_S$ be the equivalence relation on $S$ and $R_S\subset S\times S$ be the set defined by $R_S=\{(x',y')\in S\times S; x'\sim_Sy'\}$.
Let $q:X\times X\to S\times S$ be the induced map.
Then $q$ is dominant.

We shall show $q(R)\subset R_S$.
Indeed let $q((x,y))\in q(R)$, where $(x,y)\in R$.
Then there exists a sequence $f_1,\ldots,f_l:\mathbb C\to X$ such that 
$$x\in Z_1, Z_1\cap Z_2\not=\emptyset,\ldots,Z_{l-1}\cap Z_l\not=\emptyset, y\in Z_l,$$
where $Z_i\subset X$ is the Zariski closure of $f_i(\mathbb C)\subset X$.
Then the Zariski closure of $p\circ f_i:\mathbb C\to S$ is $\overline{p(Z_i)}$.
We have
$$
p(x)\in \overline{p(Z_1)}, \overline{p(Z_1)}\cap \overline{p(Z_2)}\not=\emptyset,\ldots,\overline{p(Z_{l-1})}\cap \overline{p(Z_l)}\not=\emptyset, p(y)\in \overline{p(Z_l)}.
$$
Hence $(p(x),p(y))\in R_S$.
Thus $q(R)\subset R_S$.
Hence $q(\overline{R})\subset \overline{R_S}$.
By $\overline{R}=X\times X$, we have $\overline{R_S}=S\times S$, for $q$ is dominant.
\end{proof}

\begin{lem}\label{lem:20230407}
Let $X$ be a smooth, $h$-special quasi-projective variety and let $p:X'\to X$ be a proper birational morphism from a quasi-projective variety $X'$.
Then $X'$ is $h$-special.
\end{lem}

\begin{rem}
We can not drop the smoothness assumption for $X$.
See \Cref{ex:20230418} below.
\end{rem}

\begin{proof}[Proof of \cref{lem:20230407}]
\noindent {\em Step 1. } 
In this step, we assume that $p:X'\to X$ is a blow-up along a smooth subvariety $C\subset X$.
We first prove that every entire curve $f:\mathbb C\to X$ has a lift $f':\mathbb C\to X'$, i.e., $p\circ f'=f$.
The case $f(\mathbb C)\not\subset C$ is well-known, so we assume that $f(\mathbb C)\subset C$.
There exists a vector bundle $E\to C$ so that its projectivization $P(E)\to C$ is isomorphic to $p^{-1}(C)\to C$.
The pull-back $f^*E\to \mathbb C$ is isomorphic to the trivial vector bundle over $\mathbb C$.
Thus we may take a nowhere vanishing section of $f^*E\to \mathbb C$.
This yields a holomorphic map $f':\mathbb C\to P(E)$.
Hence $f$ has a lift $f':\mathbb C\to X'$.

Now let $\sim'$ be the equivalence relation on $X'$ and $R'\subset X'\times X'$ be the set defined by $R'=\{(x',y')\in X'\times X'; x'\sim'y'\}$.
Let $q:X'\times X'\to X\times X$ be the induced map.

To show $R\subset q(R')$, we take $(x,y)\in R$.
Then there exists a sequence $f_1,\ldots,f_l:\mathbb C\to X$ such that 
$$x\in Z_1, Z_1\cap Z_2\not=\emptyset,\ldots,Z_{l-1}\cap Z_l\not=\emptyset, y\in Z_l,$$
where $Z_i\subset X$ is the Zariski closure of $f_i(\mathbb C)\subset X$.
For each $f_i:\mathbb C\to X$, we take a lift $f_i':\mathbb C\to X'$.
Let $Z_i'\subset X'$ be the Zariski closure of $f_i'(\mathbb C)\subset X'$.
Then the induced map $Z_i'\to Z_i$ is proper surjective.

For each $i=1,2,\ldots,l-1$, we take $z_i\in Z_i\cap Z_{i+1}$.
We define a holomorphic map $\varphi_i:\mathbb C\to X'$ as follows.
If $z_i\not\in C$, then $p^{-1}(z_i)$ consists of a single point $z_i'\in X'$ and thus $z_i'\in Z_i'\cap Z_{i+1}'$.
In particular $Z_i'\cap Z_{i+1}'\not=\emptyset$.
In this case, we define $\varphi_i$ to be the constant map such that $\varphi_i(\mathbb C)=\{z_i'\}$.
If $z_i\in C$, we have $p^{-1}(z_i)=\mathbb P^d$, where $d=\mathrm{codim}(C,X)-1$.
We have $Z_i'\cap p^{-1}(z_i)\not=\emptyset$ and $Z_{i+1}'\cap p^{-1}(z_i)\not=\emptyset$.
In this case, we take $\varphi_i:\mathbb C\to p^{-1}(z_i)$ so that the image is Zariski dense in $p^{-1}(z_i)$.
We define $W_i\subset X'$ to be the Zariski closure of $\varphi_i(\mathbb C)\subset X'$.
Then we have
$$
Z_1'\cap W_1\not=\emptyset, W_1\cap Z_2'\not=\emptyset, Z_2'\cap W_2\not=\emptyset, \ldots, Z_{l-1}'\cap W_{l-1}\not=\emptyset, W_{l-1}\cap Z_l'\not=\emptyset.
$$
We take $a\in Z_1'$ and $b\in Z_l'$ such that $p(a)=x$ and $p(b)=y$.
Then $q((a,b))=(x,y)$ and $(a,b)\in R'$.
Thus $R\subset q(R')$.
This induces $\overline{R'}= X'\times X'$ as in the proof of \cref{lem:202304061}.

\medskip

\noindent {\em Step 2. } 
We consider the general proper birational morphism $p:X'\to X$.  
 Then we can apply a theorem of Hironaka 
 (cf. \cite[Corollary 3.18]{Kol07}) 
to $p^{-1}:X\dashrightarrow X'$ to conclude 
there exists a sequence of blowing-ups
$$
X_k\overset{\psi_k}{\longrightarrow}X_{k-1}\overset{\psi_{k-1}}{\longrightarrow}X_{k-2}\longrightarrow\cdots\longrightarrow X_1\overset{\psi_1}{\longrightarrow}X_0=X$$
such that
\begin{itemize}
\item
each $\psi_i:X_{i}\to X_{i-1}$ is a blow-up along a smooth subvariety of $X_{i-1}$, and
\item
there exists a morphism $\pi:X_k\to X'$ such that $p\circ\pi=\psi_1\circ\psi_2\circ\cdots\circ\psi_k$.
\end{itemize}
Then by the step 1, each $X_i$ is $h$-special.
In particular, $X_k$ is $h$-special.
Thus by \cref{lem:202304063}, $X'$ is $h$-special.
\end{proof}

\begin{lem}\label{lem:202304065}
If a positive dimensional quasi-projective variety $X$ is pseudo Brody hyperbolic, then $X$ is not $h$-special.
\end{lem} 
 \begin{proof}
  We take a proper Zariski closed subset $E\subsetneqq X$ such that every non-constant holomorphic map $f:\mathbb C\to X$ satisfies $f(\mathbb C)\subset E$.
Let $x\in X$ satisfies $x\not\in E$.
Assume that $y\in X$ satisfies $x\sim y$.
Then there exists a sequence $f_1,\ldots,f_l:\mathbb C\to X$ such that 
$$x\in Z_1, Z_1\cap Z_2\not=\emptyset,\ldots,Z_{l-1}\cap Z_l\not=\emptyset, y\in Z_l,$$
where $Z_i\subset X$ is the Zariski closure of $f_i(\mathbb C)\subset X$.
Then $f_1$ is constant map and $Z_1=\{x\}$.
By $Z_1\cap Z_2\not=\emptyset$, we have $x\in Z_2$.
This yields that $f_2$ is constant and $Z_2=\{x\}$.
Similarly, we have $Z_3=\cdots=Z_l=\{ x\}$.
Hence $y=x$.
Thus we have $R\subset (X\times E)\cup (E\times X)\cup \Delta$, where $\Delta\subset X\times X$ is the diagonal.
Hence $R\subset X\times X$ is not Zariski dense. 
\end{proof}

\begin{cor}\label{cor:202304071}	 	
 	Let $X$ be a  complex   normal quasi-projective  
	variety and let $G$ be a semisimple algebraic group over $\bC$. 
 If  $\varrho:\pi_1(X)\to G(\bC)$ is a Zariski dense representation, then there  exist
 a  finite \'etale cover \(\nu:\widehat{X}\to X\), a birational and proper morphism \(\mu:\widehat{X}'\to \widehat{X}\), a dominant morphism $f:\widehat{X}'\to Y$  with connected general fibers, and a   big   representation \(\tau : \pi_{1}(Y) \to G(\bC)\)  such that
 \begin{itemize}
 	\item   \(f^{\ast} \tau = (\nu\circ\mu)^{\ast}\varrho\). 
 	\item There is a proper Zariski closed subset $Z\subsetneqq Y$ such that any closed   subvariety of $Y$ not contained in $Z$ is  of log general type.  
 	\item  $Y$ is pseudo Picard hyperbolic, and in particular pseudo Brody hyperbolic.  
 \end{itemize}  
Specifically,   $X$ is not weakly special and does not contain Zariski dense entire curves.  Furthermore, if $X$ is assumed to be smooth, then it cannot be $h$-special.
 \end{cor}
 \begin{proof} 
 	By \cref{lem:kollar}, there exists a commutative diagram of quasi-projective varieties
 	\[
 	\begin{tikzcd}
 		\widehat{X}' \arrow[r, "\mu"] \arrow[d, "f"] & \widehat{X} \arrow[r, "\nu"] & X \\
 		Y
 	\end{tikzcd}
 	\]
 	where \(\nu\) is finite \'etale, \(\mu\) is birational and proper, and a big representation \(\tau : \pi_{1}(Y) \to G(\bC)\) such that \(f^{\ast} \tau = (\nu\circ\mu)^{\ast}\varrho\). Moreover, $\widehat{X}'$ and $Y$ are smooth. 
 	Since $\varrho$ is Zariski dense, so is $(\nu \circ \mu)^*\varrho$ for $(\nu \circ \mu)_*(\pi_1(\widehat{X}'))$ is a finite index subgroup in $\pi_1(X)$. Thus, since the image $\tau(\pi_1(Y))$ includes that of $f^*\tau=(\nu\circ\mu)^*\varrho$, it follows that $\tau$ is also Zariski dense. By \cref{main2}, $Y$ is of log general type and pseudo-Picard hyperbolic, implying that $X$ is not weakly special. If there is a Zariski dense entire curve $\gamma:\bC\to X$, it can be lifted to $\gamma':\bC\to \widehat{X}'$, from which $f\circ\gamma':\bC\to Y$ would be a Zariski dense entire curve due to $f$ being dominant. However, this leads to a contradiction, thereby indicating that $X$ does not admit Zariski dense entire curves.
 	
 	Assuming $X$ is smooth, let us suppose that $X$ is $h$-special. 
By \cref{lem:202304061}, $\widehat{X}$ is $h$-special.
Moreover $\widehat{X}$ is smooth.
Hence by \cref{lem:20230407}, $\widehat{X}'$ is $h$-special.
Hence by \cref{lem:202304063}, $Y$ is $h$-special.
This contradicts to \cref{lem:202304065}.
Hence $X$ is not $h$-special.
 \end{proof}

\begin{example}\label{ex:20230418}
There exists a singular, normal projective surface $X$ such that
\begin{itemize}
\item $X$ is not weakly special,
\item $X$ does not contain Zariski dense entire curve,
\item $X$ is $h$-special and $H$-special,
\item
there exists a proper birational modification $X'\to X$ such that $X'$ is neither $h$-special nor $H$-special.
\end{itemize}

The construction is as follows.
Let $C\subset \mathbb P^2$ be a smooth projective curve of genus greater than one.
Then $C$ is of general type and hyperbolic.
Let $p\in \mathbb P^3$ be a point and $\varphi:\mathbb P^3\backslash\{p\}\to\mathbb P^2$ be the projection from the point $p\in\mathbb P^3$.
Namely for each $y\in\mathbb P^2$, $\varphi^{-1}(y)\cup\{p\}\subset \mathbb P^3$ is a projective line  $\mathbb P^1\subset \mathbb P^3$ passing through $p\in\mathbb P^3$.
We denote this line by $\ell_y\in\mathbb P^3$.
Set $X=\varphi^{-1}(C)\cup\{p\}$, which is the cone over $C$.
Then $X$ is projective and normal.

Let $X'=\mathrm{Bl}_pX$ be the blow-up of $X$ at $p\in X$.
Then we have a morphism $\varphi':X'\to C$, that is a ruled surface on $C$. 
Since $C$ is of general type, this shows that $X'$ is not weakly special.
If $X$ contains Zariski dense entire curve $f:\mathbb C\to X$, it lifts as a Zariski dense entire curve $f':\mathbb C \to X'$.
Hence $\varphi'\circ f':\mathbb C\to C$ becomes a non-constant holomorphic map, which is a contradiction.
Hence $X$ does not contain Zariski dense entire curve.

Note that $X=\bigcup_{y\in C}\ell_y$.
Let $x,x'\in X$ be two points.
We take $y,y'\in C$ such that $x\in\ell_{y}$ and $x'\in\ell_{y'}$.
Then we have $d_X(x,p)=d_X(x',p)=0$, where $d_X$ is the Kobayashi pseudo-distance on $X$.
Hence $d_X(x,x')=0$.
This shows that $X$ is $H$-special.
There exist entire curves $f:\mathbb C\to\ell_y$ and $f':\mathbb C\to \ell_{y'}$.
By $\ell_y\cap\ell_{y'}=\{p\}$, we conclude that $X$ is $h$-special.

Now the existence of the morphism $\varphi':X'\to C$ shows that $X'$ is not $H$-special.
Since $C$ is not $h$-special, \cref{lem:202304063} shows that $X'$ is not $h$-special.
\end{example}

We give another example of Brody special quasi-projective varieties.
\begin{lem}\label{lem:ZD}
	Let $\alpha:X\to \cA$ be  a  (possibly non-proper)  morphism  from a smooth quasi-projective variety $X$ to a semi-abelian variety $\cA$ with   $\overline{\kappa}(X)=0$. 
	Assume that $\dim X=\dim \alpha(X)$ and $\dim X>0$.
	Then $X$ admits a Zariski dense entire curve $\mathbb C\to X$.
	In particular, $X$ is $h$-special. 
\end{lem}

\begin{proof} 
For the existence of a Zariski dense entire curve $\mathbb C\to X$, we refer the readers to \cite[Lemma 1.5]{CDY25}.
	By \cref{lem:20230406}, $X$ is $h$-special.
\end{proof}

%\begin{rem}
%We expect that the above $Y$ should satisfy $\dim Y\leq \rank G$.  This is proved by Corlette-Simpson \cite{CS08} when $G=SL_2$ and $\varrho$ has quasi-unipotent monodromies at infinity.  It this is true, it will imply that a big representation $\varrho:\pi_1(X)\to G(\bC)$ exist only if $\dim X\leq \rank G$.
%\end{rem}

\section{Fundamental groups of special varieties}\label{sec:VN}
In this section we  prove  the linear version of \Cref{conj:revised2}. 
\begin{thm}[=\Cref{main5}]\label{thm:VN}
	Let $X$ be a   special  or $h$-special smooth quasi-projective variety.   Let  $\varrho:\pi_1(X)\to {\rm GL}_N(\bC)$ be a  linear representation.   
	Then the identity component of the Zariski closure of the image $\varrho(\pi_1(X))$  decomposes as a direct product $U\times T$, where $U$ is a  unipotent algebraic group, and $T$ is an algebraic torus. 
	In particular,  $\varrho(\pi_1(X)) $ is {virtually nilpotent}.  If $\varrho$ is reductive,  then  $\varrho(\pi_1(X)) $ is  {virtually abelian}.
\end{thm} 
It is indeed based on the following  theorem. 
\begin{thm}\label{thm:202210123}
Let $X$ be a special or $h$-special smooth quasi-projective variety.
Let $G$ be a connected, solvable linear algebraic group defined over $\mathbb C$.
Assume that there exists a Zariski dense representation $\varphi:\pi_1(X)\to G$.
Then $G$ decomposes as a direct product $U\times T$, where $U$ is a  unipotent algebraic group, and $T$ is an algebraic torus.  
In particular, $G$ is nilpotent.
\end{thm} 
The proof consists of the following three inputs: 
\begin{itemize} 
\item
Algebraic property of solvable algebraic groups (\cref{claim:202210121}).
\item
$\pi_1$-exactness of quasi-Albanese map for  special or $h$-special smooth quasi-projective variety (\cref{pro:202210131}).
\item
Deligne's unipotency theorem for monodromy action.
\end{itemize}
\begin{rem}
	Note that when $X$ is compact K\"ahler,  \cref{thm:202210123} is proved by Campana  \cite{Cam01} and Delzant \cite{Del10}.  Our proof of \cref{thm:202210123} is inspired by \cite[\S 4]{Cam01}.
\end{rem}

The structure of this section is organized as follows. 
In \cref{sec:qA} we prove a structure theorem for the quasi-Albanese morphism of  $h$-special or weakly special smooth quasi-projective varieties. \cref{s1,s2,s3,s4} are devoted to the proof of \cref{thm:202210123}. In \cref{s5} we prove \cref{thm:VN}. The last section is  on some examples of complex manifolds admitting dense entire curves.

\subsection{Structure of the quasi-Albanese morphism}\label{sec:qA}

\begin{lem}\label{prop:factor}
	Let $X$ be an $h$-special or weakly special smooth quasi-projective variety.  Then  the quasi-Albanese morphism $\alpha:X\to \cA$ is dominant with connected general fibers. 
\end{lem}
\begin{proof} 
	%Denote by $T$ the Zariski closure of the image $\alpha(X)$.  Consider the composed map $\alpha\circ f$.  Then $T$ is the Zariski closure $\alpha\circ f$. On the other hand,  by the logarithmic Bloch-Ochiai theorem (see \cite[Theorem 6.2.1]{NW13}) $T$ is a translate of semi-abelian variety.    

We first assume that $X$ is $h$-special.
	Let $\beta:X\to Y$  and $g:Y\to \cA$ be the quasi-Stein factorisation of $\alpha$ in \cref{lem:Stein}. 
Then $Y$ is $h$-special (cf. \cref{lem:202304063}) and  $\overline{\kappa}(Y)\geq 0$. 
To show $\overline{\kappa}(Y)=0$, we assume contrary that $\overline{\kappa}(Y)>0$.  By a theorem of Kawamata \cite[Theorem 27]{Kaw81},  there are a semi-abelian variety $B\subset \cA$,   finite \'etale Galois covers $\widetilde{Y}\to Y$ and $\widetilde{B}\to B$, and a normal algebraic variety $Z$ such that
	\begin{itemize}
		\item there is a finite morphism from $Z$ to the quotient $\cA/B$;
		\item  $\widetilde{Y}$ is a fiber bundle over $Z$ with fibers $\widetilde{B}$;
		\item $\overline{\kappa}(Z)=\dim Z=\overline{\kappa}(Y)$.
	\end{itemize} 	  
By \cref{lem:202304061}, $\widetilde{Y}$ is $h$-special.
Hence by \cref{lem:202304063}, $Z$ is $h$-special.
Thus by \cref{lem:202304065}, $Z$ is not pseudo-Brody hyperbolic.
On the other hand, by \cite[Theorem C]{CDY25}, $Z$ is pseudo-Brody hyperbolic, a contradiction. 
Hence $\overline{\kappa}(Y)=0$.

	Assume now $X$ is weakly special.  Consider a connected component $\widetilde{X}$ of $X\times_Y\widetilde{Y}$. Then $\widetilde{X}\to  X$ is finite \'etale.  The composed morphism $\widetilde{X}\to Z$ of $\widetilde{X}\to \widetilde{Y}$ and $\widetilde{Y}\to Z$ is dominant with connected general fibers. We obtain a contradiction since $Z$ is of log general type and $X$ is weakly special. Hence $\overline{\kappa}(Y)=0$. 
	
	By \cite[Theorem 26]{Kaw81}, $Y$ is a semi-abelian variety and according to the universal property of quasi-Albanese morphism, $Y=\cA$. Hence $\alpha$ is dominant with  connected general fibers.  
\end{proof}

\subsection{A nilpotency condition for solvable linear group} \label{s1}
All the ground fields for algebraic groups and linear spaces are $\mathbb C$ in this subsection.

\begin{lem}  \label{lem:trivial morphism}
		Let $T$ be an algebraic torus and let $U$ be a unipotent group.
		Then every morphism $f:T\to U$ of algebraic groups is constant.
	\end{lem}
\begin{proof}
		Since $U$ is unipotent, we have a sequence $\{e\}=U_0\subset U_1\subset \cdots\subset U_n=U$ of normal subgroups of $U$ such that $U_k/U_{k-1}=\mathbb G_a$, where $\mathbb G_a$ is the additive group.
		To show $f(T)\subset U_0=\{e\}$, suppose contrary $f(T)\not\subset U_0$.
		Then we may take the largest $k$ such that $f(T)\not\subset U_k$.
		Then $k<n$ and $f(T)\subset U_{k+1}$.
		Thus we get a non-trivial morphism $T\to U_{k+1}/U_k=\mathbb G_a$.
		By taking a suitable subgroup $\mathbb G_m\subset T$, we get a non-trivial morphism $g:\mathbb G_m\to \mathbb G_a$. 
		But this is impossible.
		Indeed let $\mu=\cup_n\mu_n$, where $\mu_n=\{ a\in \mathbb C^*;a^n=1\}$.
		Then $|\mu|=\infty$.
		On the other hand, for $a\in \mu_n$, we have $ng(a)=g(a^n)=g(1)=0$, hence $g(a)=0$.
		Thus $\mu\subset g^{-1}(0)$.
		This is impossible since we are assuming that $g$ is non-constant. 
		\end{proof}

\begin{lem}\label{lem:20221012}
Let $T$ be an algebraic torus.
Let $0\to L'\to L\to L''\to 0$ be an exact sequence of vector spaces with equivariant $T$-actions.
If $L'$ and $L''$ have trivial $T$-actions, then $L$ has also a trivial $T$-action.
\end{lem}

\begin{proof}
Let $\varphi:T\to \mathrm{GL}(L)$ be the induced morphism of algebraic groups.
We take a (non-canonical) splitting $L=L'\oplus L''$ of vector spaces.
Let $t\in T$.
For $(v',0)\in L'$, we have $(\varphi(t)-\mathrm{id}_L)v'=0$, for $T$ acts trivially on $L'\subset L$.
For $(0,v'')\in L''$, we have $(\varphi(t)-\mathrm{id}_L)v''\in L'$.
Hence for $(v',v'')\in L$ and $t\in T$, we have 
$$(\varphi(t)-\mathrm{id}_L)^2\cdot (v',v'')=(\varphi(t)-\mathrm{id}_L)\cdot (u,0)=0,$$ where $u=(\varphi(t)-\mathrm{id}_L)v''\in L'$.
Hence $\varphi:T\to \mathrm{GL}(L)$ factors through the unipotent group $U\subset \mathrm{GL}(L)$.
Since the map $T\to U$ is trivial (cf. \cref{lem:trivial morphism}), 
$L$ has a trivial $T$-action.
\end{proof}

Let $G$ be a connected, solvable linear group.
We have an exact sequence 
\begin{equation}\label{eqn:20230321}
1\to U\to G\to T\to 1,
\end{equation}
where $U=R_u(G)$ is the unipotent radical and $T\subset G$ is a maximal torus.
Then $T$ acts on $U/U'$ by  conjugation.
The following lemma is from \cite[Lemma 1.8]{AN99}. 

\begin{lem}\label{eqn:202210154}
If $T$ acts trivially on $U/U'$, then $G$ decomposes as a direct product $G \cong U \times T$.
In particular, $G$ is nilpotent. 
\end{lem}

\begin{proof}
By $T\subset G$, the conjugation action yields a $T$-action on $U$.
We shall show that this $T$-action is trivial.
To show this, we set $N=\mathrm{Lie}(U)$.
Then the $T$-action yields $\alpha:T\to \mathrm{Aut}(N)$.
We set 
$$S=\{ \sigma\in \mathrm{Aut}(N); (\mathrm{id}_N-\sigma)N\subset [N,N]\}.$$
Then the elements of $S$ are unipotent.
Since $T$ acts trivially on $U/U'$, we have $\alpha(T)\subset S$.
Hence $\alpha$ is trivial (cf \cref{lem:trivial morphism}).
Hence $T$ acts trivially on $U$.
Thus the elements of $T$ and $U$ commute.
Thus $G=U\times T$.
Since $U$ is nilpotent, $G$ is nilpotent.
\end{proof}

Since $T$ is commutative, we have 
$G'\subset U$, where $G'=[G,G]$ is the commutator subgroup.
Hence we have
$$
1\to U/G'\to G/G'\to T\to 1. 
$$
Since $G/G'$ is commutative and $U/G'$ is unipotent, we have $G/G'=(U/G')\times T$. 
By
$$
1\to G'/G''\to G/G''\to G/G'\to 1,
$$
$G/G'$ acts on $G'/G''$ by conjugation.
By $T\subset (U/G')\times T=G/G'$, we get $T$-action on $G'/G''$.

\begin{lem}\label{claim:202210121}
Assume $T$ acts trivially on $G'/G''$.
Then $G$ decomposes as a direct product $G \cong U \times T$.
In particular, $G$ is nilpotent. 
\end{lem} 

\begin{proof}
By $G'\subset U$, we have $G''\subset U'$, where $U'=[U,U]$.
Then we get
$$
1\to U'/G''\to G'/G''\to G'/U'\to 1.
$$
Hence $T$ acts trivially on $G'/U'$.
By $G/G'\simeq T\times (U/G')$, we note that $T$ acts trivially on $U/G'$.
Now we have the following exact sequence
$$
1\to G'/U'\to U/U'\to U/G'\to 1.
$$
This induces
$$
0\to \mathrm{Lie}(G'/U')\to \mathrm{Lie}(U/U')\to \mathrm{Lie}(U/G')\to 0.
$$
We know that both $T\to \mathrm{GL}(\mathrm{Lie}(G'/U'))$ and $T\to \mathrm{GL}(\mathrm{Lie}(U/G'))$ are trivial.
Hence by Lemma \ref{lem:20221012}, $T$ acts trivially on $\mathrm{Lie}(U/U')$.
Hence $T$ acts trivially on $U/U'$.
By \cref{eqn:202210154}, $G$ decomposes as a direct product $G \cong U \times T$.
\end{proof}

\subsection{A lemma on Kodaira dimension of fibration}
 In this section, we prove an extension to the quasi-projective setting of \cite[Theorem 1.8]{Cam04}, and derive a useful consequence concerning algebraic fiber spaces of special varieties over semi-abelian varieties (see Corollary~\ref{cor:specialalg}).

We refer to  \Cref{sec:spechspecial} and to \cite{Cam11} for several definitions concerning orbifold bases. For our purposes, we only need to recall the following.
\medskip

Let \(f : X \to Y\) be a dominant morphism with general connected fibers between smooth quasi-projective varieties. Let $\Delta \subset Y$ be a prime divisor. We denote by $J(\Delta )$ the set of all prime divisors $D$ in $X$ such that $\overline{f(D)}=\Delta$. 
We write 
\begin{equation}\label{eqn:div}
f^*(\Delta )=\sum_{j\in J(\Delta )}m_jD_j+E_{\Delta}
\end{equation}
where $E_{\Delta}$ is an exceptional divisor of $f$, i.e., the codimension of $f(E_{\Delta})$ is greater than one.
We put 
$$m_{\Delta}=
\begin{cases}
\min_{j\in J(\Delta )}m_j & \text{if $J(\Delta)\not=\emptyset$} \\
+\infty & \text{if $J(\Delta)=\emptyset$}
\end{cases}
$$
Let $I(f)$ be the set of all $\Delta$ such that $m_{\Delta}\geq 2$. 
The {\em orbifold ramification divisor} of \(f\) on \(Y\) is $\Delta (f)=\sum_{\Delta \in I(f)} \left( 1 - \frac{1}{m_{\Delta}} \right)\Delta$.
 Consider now \(\overline{f} : \overline{X} \to \overline{Y}\), any extension of \(f\) between log-smooth compactifications, and let \(D := \overline{X} - X\) and \(G := \overline{Y} - Y\). Then the orbifold base of \(\bar{f}: (\overline{X}| D)\to \overline{Y}\)   is given by \(\Delta(\bar{f}, D) := G + \overline{\Delta(f)}\) (where \(\overline{\Delta(f)}\) is the \(\mathbb{Q}\)-divisor on \(\overline{Y}\) obtained from \(\Delta(f)\) by taking the closure of every component).

\begin{proposition}\label{prop:KDF}  Let $\{f_i:X_i\to Y_i\}_{i=1,2}$   be  dominant morphisms between smooth quasi-projective varieties with connected general fibers, and let \(\{\overline{f}_{i}:\overline{X}_{i}\to \overline{Y}_{i}\}_{i=1,2}\) be two extensions to log-smooth projective compactifications. We assume that we have two compatible commutative diagrams
	\begin{equation*}
		\begin{tikzcd}
			X_2 \arrow[r, "u_\circ"]\arrow[d, "f_2"] & X_1\arrow[d, "f_1"] \\
			Y_2 \arrow[r, "v_\circ"] & Y_1 
		\end{tikzcd}
		\hspace*{2em}
		\text{and}
		\hspace*{2em}
		\begin{tikzcd}
			\overline{X}_2 \arrow[r, "u"]\arrow[d, "\overline{f}_2"] & \overline{X}_1\arrow[d, "\overline{f}_1"] \\
			\overline{Y}_2 \arrow[r, "v"] & \overline{Y}_1 
		\end{tikzcd}
	\end{equation*}
	where $u_\circ$ and $v_\circ$ are  proper birational morphisms. Denoting by \(D_{i} := \overline{X}_{i} - X_{i}\), we let \((\overline{Y}_{i}| \Delta(\bar{f}_{i}, D_{i}))\) (resp. \((Y| \Delta(f_{i}))\) be the orbifold base of the fibration \(\overline{f}_{i} : (\overline{X}_{i}| D_{i}) \to \overline{Y}_{i}\) (resp. of the fibration \(f_{i} : X_{i} \to Y_{i}\)). Then, one has
	\begin{thmlist}
	\item the orbifold Kodaira dimension satisfy \(\kappa(\overline{Y}_{2}| \Delta(\bar{f}_{2}, D_{2})) \leq \kappa(\overline{Y}_{1}| \Delta(\bar{f}_{1}, D_{1}))\);  
\item If \(\bar{\kappa}(Y_{1}) \geq 0\) and $\bar{\kappa}(\overline{Y}_1-{\rm Supp}\, \Delta(\bar{f}_1,D_1))=\dim Y_1$, then   
 $K_{\overline{Y}_1}+\Delta(\bar{f}_1,D_1)$ and  $ K_{\overline{Y}_2}+\Delta(\bar{f}_2,D_2)$ are both big $\mathbb Q$-divisors.  
 In particular,  $\kappa(Y_1,f_1)=\dim Y_1$, where $\kappa(Y_1,f_1)$ is the Kodaira dimension of   $f_1:X_1\to Y_1$ defined in  \eqref{eq:Kodaira2}.
	\end{thmlist}
\end{proposition}

\begin{proof}[Proof of \cref{prop:KDF}]
Let us prove the first statement.   By \cite[Lemme 4.6]{Cam11},  we have  
\begin{align} \label{eq:Cam}
	v_*\Delta(\bar{f}_{2}, D_{2})=\Delta(\bar{f}_{1}, D_{1}).
\end{align}
As $v:\overline{Y}_2\to\overline{Y}_1$ is a birational morphism of smooth projective varieties, we
also have $v_*K_{\overline{Y}_2}=K_{\overline{Y}_1}$, which implies that 
\begin{align} \label{eq:pushforward}
	v_*\big(K_{\overline{Y}_2}+\Delta(\bar{f}_{2}, D_{2})\big)
	=K_{\overline{Y}_1}+\Delta(\bar{f}_{1}, D_{1}).
\end{align}
It follows that, 
\begin{align} \label{eq:Cam2}
K_{\overline{Y}_2}+\Delta(\bar{f}_{2}, D_{2})+E_2
	=	v^*\big(K_{\overline{Y}_1}+\Delta(\bar{f}_{1}, D_{1})\big)+E_1.
\end{align} 
where $E_1$ and $E_2$ are some effective $\bQ$-divisor which is $v$-exceptional.    
This implies that for sufficiently divisible $m>0$,  we have a natural inclusion
$$
H^0\big(\overline{Y}_2, m(K_{\overline{Y}_2}+\Delta(\bar{f}_{2}, D_{2}))\big)\hookrightarrow H^0\big(\overline{Y}_2, 	v^*m\big(K_{\overline{Y}_1}+\Delta(\bar{f}_{1}, D_{1})\big)+mE_1\big).
$$
By \cite[Example 2.1.16]{Laz04}, we have
$$
H^0\big(\overline{Y}_2, 	v^*m\big(K_{\overline{Y}_1}+\Delta(\bar{f}_{1}, D_{1})\big)+mE_1\big) \simeq H^0\big(\overline{Y}_1, 	 m (K_{\overline{Y}_1}+\Delta(\bar{f}_{1}, D_{1}))\big).
$$ 
This implies the first claim. 
\medskip

	We will prove the second claim. We let \(G_{i} := \overline{Y}_{i} - Y_{i}\).     
	We note that 
	$\Delta(\bar{f}_i,D_i)=G_i+\Delta_i$, where $\Delta_i$ is an effective $\bQ$-divisor such that each of its irreducible component is not contained in $G_i$ and $\Delta_i$ does not have common components with $G_i$.   Note that the coefficient  of each irreducible component  in $\Delta_i$ lies in $(0,1]$.   Hence $\Delta_{i,{\rm red}}-\Delta_i$ is effective. 
	 Write $Y_1^\circ:=\overline{Y}_1-{\rm Supp}\, \Delta(\bar{f}_1,D_1)$ 
	
	 By our assumption, $ \bar{\kappa}(Y_1^\circ)=\dim Y_1$.  Therefore, 
   $
	K_{\overline{Y}_1}+\Delta_{1,{\rm red}}+G_1
	$  is a   big line bundle by \cite[Lemma 3]{NWY13}.   Since we assume that $\kappa(K_{\overline{Y}_1})\geq 0$, it follows that 
	$$
m	K_{\overline{Y}_1}+\Delta_{1,{\rm red}}+G_1
	$$
	is also big for any positive integer $m$. We choose $m$ sufficiently large such that 
	$$
\frac{1}{m}	\Delta_{1,{\rm red}}\leq \Delta_1. 
	$$
	Hence, $$
 	K_{\overline{Y}_1}+ \Delta(\bar{f}_{1}, D_{1}) =	K_{\overline{Y}_1}+\frac{1}{m}	(\Delta_{1,{\rm red}}+G_1)+(\Delta_1-\frac{1}{m}	\Delta_{1,{\rm red}})+\frac{m-1}{m}G_1
	$$
	is big.  
	
	Let $E$ be the exceptional divisor of $\nu$, that is a reduced effective divisor on $\overline{Y}_2$. Since $\overline{Y}_1$ and $\overline{Y}_2$  are both  smooth, it follows that
	$$
K_{	\overline{Y}_2}-E\geq \nu^*K_{	\overline{Y}_1},
	$$
	and thus
	$$
	\kappa(K_{	\overline{Y}_2}-E)\geq 0. 
	$$
	By \eqref{eq:Cam},  for any sufficiently positive integer $\ell$,  we have
	$$
 \Delta(\bar{f}_{2}, D_{2})+\ell  E\geq \nu^*\Delta(\bar{f}_{1}, D_{1}) .
	$$
 It follows that
	$$
\ell   K_{	\overline{Y}_2}+\Delta(\bar{f}_{2}, D_{2}) \geq  \nu^*(\ell K_{\overline{Y}_1}+\Delta(\bar{f}_{1}, D_{1}) )
	$$
	which is big. Hence  
	$$
 K_{	\overline{Y}_2}+\Delta(\bar{f}_{2}, D_{2}) = (K_{	\overline{Y}_2}+\frac{1}{\ell}\Delta(\bar{f}_{2}, D_{2})) +\frac{\ell-1}{\ell} \Delta(\bar{f}_{2}, D_{2})
	$$
	is also big. 
	  We thus proved the second claim. The last claim follows from the very definition of Kodaira dimension of fibration in \eqref{eq:Kodaira2}. 
\end{proof}
 
 We have the following consequence. 
 \begin{cor}\label{cor:specialalg}
Let \(X\) be a special smooth quasi-projective variety,
let \(A\) be a positive-dimensional semi-abelian variety,
and let \(p:X\to A\) be a dominant morphism with connected general fibers.
Let \(\Delta(p)\) be the orbifold base divisor on \(A\) defined at the beginning of this section, and let \(\mathrm{St}(\Delta(p))\) be the  stabilizer of the reduced divisor $\Delta(p)_{\rm red}$ under the action of \(A\). 
Then
\[
\dim \mathrm{St}(\Delta(p))>0.
\]
\end{cor}

\begin{proof}
	Let $\overline{A}$  be an equivariant smooth compactification of $A$ such that $G:=\overline{A}-A$ is a simple normal crossing divisor such that $K_{\overline{A}}+G=\cO_{\overline{A}}$. We take a smooth projective compactification $\overline{X}$ of $X$   such that $D:=\overline{X}-X$ is a simple normal crossing divisor and $p$ extends to a morphism $\bar{p}:\overline{X}\to \overline{A}$.   Denote by $\Delta(\bar{p},D)$ the orbifold divisor of $\bar{p}:(\overline{X}|D)\to \overline{A}$.   Let $\Delta(p)$ be the orbifold divisor of $p:X\to A$ defined at the beginning of this subsection. Then we have
	$$
	\Delta(\bar{p},D) =G+\overline{\Delta(p)}.
	$$
Hence $\overline{A}- \Delta(\bar{p},D)_{\rm red}=A-\Delta(p)_{\rm red}$.
	
Assume by contradiction that  \(\mathrm{dim}\; \mathrm{St}(\Delta(p)) = 0\),	by \cite[Proposition 5.6.21]{NW13}, the variety $A-\Delta(p)_{\rm red}$, hence  $\overline{A}-|\Delta(\bar{p},D)|$ is of log general type.  Therefore,  conditions in  \Cref{prop:KDF} are fulfilled. This implies that \(\kappa( {A},  {p}) = \dim A\), so \(X\) is not special by \cref{def:special}.
\end{proof}

\subsection{$\pi_1$-exactness of quasi-Albanese morphisms}\label{s2}
Let $X$ and $Y$ be smooth quasi-projective varieties. 
We say that a morphism $p:X\to Y$ is an algebraic fiber space 
if 
\begin{itemize}
\item
$p$ is dominant, and
\item
$p$ has general connected fibers. 
\end{itemize}

Let $p:X\to Y$ be an algebraic fiber space, and let $F$ be a general fiber of $p$ (assumed to be smooth). Then there is a natural sequence of morphisms of groups:
\begin{equation}\label{eqn:gp}
\pi _1(F)\overset{i_*}{\to}\pi _1(X)\overset{p_*}{\to} \pi _1(Y).
\end{equation}
Note that $p_*$ is surjective (because of the second condition of the algebraic fiber spaces) and that the image of $i_*$ is contained in the kernel of $p_*$.
\begin{dfn}
For an algebraic fiber space $p:X\to Y$, we say that $p$ is $\pi _1$-exact if the above sequence \eqref{eqn:gp} is exact.
\end{dfn}

By   \cref{prop:factor}, if $X$ is $h$-special, then the quasi-Albanese map $X\to A(X)$ is an algebraic fiber space.

 In this subsection, we prove the following proposition.
 \begin{proposition}\label{pro:202210131}
Let $X$ be a $h$-special or special smooth quasi-projective variety. 
Let $p:X\to A$ be an algebraic fiber space where $A$ is a semi-abelian variety. 
Then $p$ is $\pi _1$-exact. 
In particular, the quasi-Albanese map $a_X:X\to A(X)$ is $\pi _1$-exact.
\end{proposition}

The rest of this subsection is devoted to the proof of \cref{pro:202210131}, for which we need four lemmas. 
\begin{lem}\label{lem:3}
Let $X$ be a $h$-special or special quasi-projective variety. 
Let $A$ be a semi-abelian variety and let $p:X\to A$ be an algebraic fiber space. If $\Delta (p)\not= \emptyset$, then $\dim \mathrm{St}(\Delta (p))> 0$, where $ \mathrm{St}(\Delta (p))=\{ a\in A; \ a+\Delta(p)_{\rm red}=\Delta(p)_{\rm red}\}$.
\end{lem}

\begin{proof}
Suppose $\Delta (p)\not= \emptyset$.
To show $\dim \mathrm{St}(\Delta (p))> 0$, we assume contrary $\dim \mathrm{St}(\Delta (p))= 0$.

The case where \(X\) is special has been dealt with in \cref{cor:specialalg}, which implies readily that $\dim \mathrm{St}(\Delta (p))> 0$ in this situation. 

Next, we consider the case that $X$ is $h$-special.
Let $E\subset X$ be the exceptional divisor of $p$.
Let $Z\subset A$ be the Zariski closure of $p(E)$.
Then $\mathrm{codim}(Z,A)\geq 2$. 
We apply \cite[Proposition 10.9]{CDY25} for the reduced divisor $\Delta(p)_{\rm red}\subset A$ and $Z\cap \Delta(p)_{\rm red}$ to get a proper Zariski closed set $\Xi\subsetneqq A$.
Let $f:\mathbb C\to X$ be a holomorphic map such that $p\circ f$ is non-constant.
We first show that $p\circ f(\mathbb C)\subset \Xi\cup \Delta(p)_{\rm red}$.
If $p\circ f(\mathbb C)\subset \Delta(p)_{\rm red}$, there is nothing to prove.
Hence we may assume that
$p\circ f(\mathbb C)\not\subset \Delta(p)_{\rm red}$.
In this case, we have $\mathrm{ord}_y(p\circ f)^{*}(\Delta(p)_{\rm red})\geq 2$ for all $y\in (p\circ  f)^{-1}(\Delta(p)_{\rm red}\backslash Z)$.
Let $g:\mathbb C\to A$ be defined by $g(z)=p\circ f(e^z)$.
Then $g$ has essential singularity over $\infty$.
Thus by \cite[Proposition 10.9]{CDY25}, we have $g(\mathbb C)\subset \Xi\cup \Delta(p)_{\rm red}$.
Thus $p\circ f(\mathbb C)\subset \Xi\cup \Delta(p)_{\rm red}$.

Now since $X$ is $h$-special, we may take two points $x,y\in X$ so that $x\sim y$, $p(x)\not=p(y)$ and $p(x)\not\in \Xi\cup \Delta(p)_{\rm red}$.
Then there exists a sequence $f_1,\ldots,f_l:\mathbb C\to X$ such that 
$$x\in F_1, F_1\cap F_2\not=\emptyset,\ldots,F_{l-1}\cap F_l\not=\emptyset, y\in F_l,$$
where $F_i\subset X$ is the Zariski closure of $f_i(\mathbb C)\subset X$.
By $p(x)\not\in \Xi\cup \Delta(p)_{\rm red}$, we have $p\circ f_1(\mathbb C)\not\subset  \Xi\cup \Delta(p)_{\rm red}$.
Hence $p\circ f_1$ is constant and $F_1\subset  p^{-1}(p(x))$.
By $F_1\cap F_2\not=\emptyset$, we have $p(x)\in p(F_2)$.
Hence $p\circ f_2(\mathbb C)\not\subset  \Xi\cup \Delta(p)_{\rm red}$.
Hence $p\circ f_2$ is constant and $F_2\subset  p^{-1}(p(x))$.
Similarly, we get $F_i\subset  p^{-1}(p(x))$ for all $i=1,2,\ldots,l$ inductively.
In particular, we have $F_l\subset  p^{-1}(p(x))$.
Thus $p(y)=p(x)$.
This is a contradiction.
Thus we have proved $\dim \mathrm{St}(\Delta (p))> 0$.
\end{proof}

\begin{lem}\label{lem:4}
Let $p:X\to Y$ be an algebraic fiber space. Assume that $\Delta (p)=\emptyset$. Then $p$ is $\pi _1$-exact. 
\end{lem}

\begin{proof}
This lemma is proved implicitly in  \cite[A.C.10]{Cam98} 
and  \cite[Lemme 1.9.9]{Cam01}. 
Let $Y^*$ be a Zariski open subset of $Y$ on which $p$ is a locally trivial $C^\infty$ fibration. 
We may assume that $Y\backslash Y^*$ is a divisor. Set $X^*=p^{-1}(Y^*)$. For each irreducible component $\Delta $ of $Y\backslash Y^*$, we write $p^*(\Delta )=\sum_{j\in J(\Delta )}m_jD_j+E_{\Delta}$ as in \eqref{eqn:div}. Note that $J(\Delta)$ is non-empty, or else we will have $\Delta\leq \Delta(p)$. 
Then since $m_{\Delta}=1$, there is a divisor $D_j$, $j\in J(\Delta )$, such that $m_j=1$. We denote this $D_j$ by $D_{\Delta }$.
Fix a base point $x$ in $X^*$ and let $\gamma _{\Delta}$ be a small loop going once around the divisor $D_{\Delta}$ in the counterclockwise direction. Then $\delta _{\Delta}=p_*(\gamma _{\Delta})$ is a small loop going once around the divisor $\Delta$ for the corresponding base point $y=p(x)$ in $Y^*$.
Since the restriction $p|_{X^*}:X^*\to Y^*$ is smooth and surjective, 
we have the following exact sequence:
\begin{equation}\label{eqn:1}
\pi _1(F,x)\to \pi _1(X^*,x)\to \pi _1(Y^*,y)\to 1,
\end{equation}
where $F$ is the fiber of $p$ over $y$. Let $\Phi$ be the image of $\pi _1(F,x)\to \pi _1(X^*,x)$.
\par
Now we look the following exact sequence:
\begin{equation*}
\begin{CD}
@. @. 1 @. 1\\
@. @. @VVV @VVV \\
@. @. \Phi @>>> \Psi @>>> \widetilde{C}@>>> 1\\
@. @. @VVV @VVV \\
1 @>>> K @>>> \pi _1(X^*,x) @>>> \pi _1(X,x)@>>> 1\\
@. @VVV @VVV @VVV @.\\
1 @>>> L @>>> \pi _1(Y^*,y) @>>> \pi _1(Y,y)@>>> 1\\
@. @VVV @VVV  @VVV\\
@. C @. 1 @. 1\\
@. @VVV\\
@. 1
\end{CD}
\end{equation*}
Note that $\widetilde{C}$ and $C$ are naturally isomorphic.
To prove our lemma, it is enough to show that $\widetilde{C}$, hence $C$ is trivial.
Note that by Van Kampen's theorem, $L$ is normally generated by $\delta _{\Delta}$. Since $\pi_1(X^*,x)\to \pi_1(Y^*,y)$ is surjective, it follows that the image of $K$ in $\pi_1(Y^*,y)$ is also normal.  Thus the map $K\to L$ is surjective, hence $C$ is trivial.
\end{proof}

Let $p:X\to Y$ and $q:Y\to S$ be algebraic fiber spaces.
Given $s\in S$, we consider $p_s:X_s\to Y_s$, where $X_s$ is the fiber of $q\circ p:X\to S$ and $Y_s$ is the fiber of $q:Y\to S$.

\begin{lem}\label{lem:20221020}
For generic $s\in S$, we have $\Delta(p_s)\subset \Delta (p)\cap Y_s$.
\end{lem}

To prove this lemma, we start from the following general discussion.

\begin{claim}\label{claim:202210221}
Let $p:V\to W$ and $q:W\to S$ be dominant morphisms of
irreducible quasi-projective varieties.
Then for generic $s\in S$, the map
$p_s:V_s\to W_s$ is dominant.
\end{claim}

\begin{proof}
Since $p$ is dominant, by Chevalley's theorem there exists a non-empty
Zariski open subset $W^o\subset W$ such that
\[
W^o\subset p(V).
\]
Set $Z:=W\setminus W^o$.
Since $q:W\to S$ is dominant, after replacing $S$ by a non-empty
Zariski open subset, we may assume that $W_s$ is non-empty for every
$s\in S$.

For every irreducible component $C$ of $W_s$, we have
\[
\dim C\geq \dim W-\dim S.
\]
Since $Z$ is a proper closed subset of $W$, we have
$\dim Z\leq \dim W-1$.
After shrinking $S$ further, we may assume that
\[
\dim Z_s\leq \dim W-\dim S-1
\]
for every $s\in S$.
Hence $Z_s$ contains no irreducible component of $W_s$.
Therefore $W^o\cap W_s$ is Zariski dense in $W_s$.

For any $w\in W^o\cap W_s$, since $W^o\subset p(V)$, there exists
$v\in V$ such that $p(v)=w$. Since
\[
q(p(v))=q(w)=s,
\]
we have $v\in V_s$. Therefore
\[
W^o\cap W_s\subset p_s(V_s).
\]
Since $W^o\cap W_s$ is dense in $W_s$, the morphism
$p_s:V_s\to W_s$ is dominant.
\end{proof}

\begin{proof}[Proof of \cref{lem:20221020}]
We start from the following observation.

\begin{claim}
For generic $s\in S$, $p_s:X_s\to Y_s$ is an algebraic fiber space.
\end{claim}
\begin{proof}
By replacing $S$ by a smaller non-empty Zariski open set, we assume that $X_s$ and $Y_s$ are smooth varieties for all $s\in S$.
We apply \cref{claim:202210221}.
Then by replacing $S$ by a smaller non-empty Zariski open set, we may assume that $p_s:X_s\to Y_s$ is dominant for every $s\in S$.
Next we take a non-empty Zariski open set $Y^o\subset Y$ such that $p^{-1}(Y^o)\to Y^o$ has connected fibers.
We take a non-empty Zariski open set $S^o\subset S$ such that $S^o\subset q(Y^o)$ and that $Y_s$ is irreducible for every $s\in S^o$.
Hence for every $s\in S^o$, $Y^o\cap Y^s\subset Y_s$ is a dense Zariski open set.
Note that $p_s:X_s\to Y_s$ has connected fibers over $Y^o\cap Y^s\subset Y_s$.
Hence for every $s\in S^o$, $p_s:X_s\to Y_s$ is an algebraic fiber space.
\end{proof}

Since the problem is local in $S$, by replacing $S$ by a smaller non-empty Zariski open set, we assume that $p_s:X_s\to Y_s$ are algebraic fiber spaces for all $s\in S$.
In particular, both $X_s$ and $Y_s$ are smooth.

\begin{claim}\label{claim:202210222}
Let $D\subset Y$ be an irreducible and reduced divisor such that $q(D)\subset S$ is Zariski dense.
Assume that $D\not\subset \Delta(p)$.
Then for generic $s\in S$,
\begin{itemize}
\item
$D_s\subset Y_s$ is a reduced divisor, and
\item
for every irreducible component $\Delta$ of $D_s$, we have $\Delta\not\subset \Delta(p_s)$.
\end{itemize} 
\end{claim}

\begin{proof}
We first prove the first assertion.
Let $D^o\subset D$ be the smooth locus of $D$.
Then $D^o\subset D$ is a dense Zariski open set.
We take a non-empty Zariski open set $S^o\subset S$ such that both $Y\to S$ and $D^o\to S$ are smooth and surjective over $S^o$.
By \cref{claim:202210221} applied to the open immersion $D^o\subset D$, we may also assume that $D^o_s\subset D_s$ is dense.
We take $s\in S^o$.
Then we have $\dim Y_s=\dim Y-\dim S$ and $\dim D_s=\dim D-\dim S$.
Hence we have $\dim D_s=\dim Y_s-1$.
Hence $D_s\subset Y_s$ is a divisor.
Note that the fiber $D^o_s=D^o\times_S\mathrm{Spec}(\mathbb C(s))$ is a regular scheme.
Hence $D_s\subset Y_s$ is a reduced divisor.
This shows the first assertion.

Next we prove the second assertion.
Let $p^*D=\sum m_jW_j+E$ be the decomposition as in \eqref{eqn:div}.
By $D\not\subset \Delta(p)$, there exists $W_{j_0}$ such that $m_{j_0}=1$.
We remove $\sum_{j\not=j_0} m_jW_j+E$ from $X$ to get a Zariski open set $Z\subset X$.
Then $Z$ is quasi-projective.
Set $W=Z\cap W_{j_0}$, which is a reduced divisor on $Z$.
Let $\pi:Z\to Y$ be the restriction of $p$ onto $Z$.
Then we have $\pi^*D=W$ and $\overline{\pi(W)}=D$.
We apply the first assertion of \cref{claim:202210222} for $W\subset X$.
Then for generic $s\in S$, $W_s\subset Z_s$ is a reduced divisor.
By \cref{claim:202210221}, $\overline{\pi_s(W_s)}=D_s$ for generic $s\in S$.
We take generic $s\in S$ and an irreducible component $\Delta$ of $D_s$.
Then by $\pi_s^*D_s=W_s$, we have $\Delta\not\subset \Delta(\pi_s)$, hence $\Delta\not\subset \Delta(p_s)$.
\end{proof}

Now let $Y^o\subset Y$ be a non-empty Zariski open set such that $p|_{X^o}:X^o\to Y^o$ is smooth and surjective, where $X^o=p^{-1}(Y^o)$.
We may assume that $D=Y-Y^o$ is a divisor.
Let $D_1,\ldots,D_n$ be the irreducible components of $D$.
By replacing $S$ by a smaller non-empty Zariski open set, we may assume that $q(D_i)$ is dense in $S$ for all $i=1,\ldots,n$.
We assume that $D_i$, $1\leq i\leq l$, are all the components of $D$ such that $D_i\not\subset \Delta(p)$.
By replacing $S$ by a smaller non-empty Zariski open set, we may assume that \cref{claim:202210222} is valid for all $s\in S$ and $D_{1},\ldots,D_l$.
We take $s\in S$.
Let $\Delta\subset Y_s$ be a prime divisor such that $\Delta\subset \Delta(p_s)$.

We first show $\Delta\subset D_s$.
For this purpose, we suppose contrary $\Delta\not\subset D_s$.
Then $\Delta\cap Y^o_s\not=\emptyset$.
Set $X^o_s=p_s^{-1}(Y^o_s)$.
Note that $p_s|_{X^o_s}:X^o_s\to Y^o_s$ is smooth and surjective.
Then the divisor $p_s^*\Delta\cap X^o_s\subset X^o_s$ is reduced and $p_s(p_s^*\Delta\cap X^o_s)=\Delta\cap Y^o_s$.
Thus $\Delta\not\subset \Delta(p_s)$, a contradiction.
Hence $\Delta\subset D_s$.

Now we may take $D_i$, $i=1,\ldots,n$, such that $\Delta\subset D_i$.
Then by \cref{claim:202210222}, we have $l+1\leq i\leq n$.
Hence by $D_i\subset \Delta(p)$ for $i=l+1,\ldots,n$, we have $\Delta\subset \Delta(p)$.
This shows $\Delta(p_s)\subset \Delta(p)\cap Y_s$.
\end{proof}

\begin{lem}\label{lem:5}
Let $A$ be a semi-abelian variety, and let $p:X\to A$ be an algebraic fiber space. Let $B$ be the quotient semi-abelian variety $A/\mathrm{St}^o(\Delta (p))$.
Let $q:X\to B$ be the composition of $p$ and the quotient $r:A\to B$. 
If $q$ is $\pi _1$-exact, then $p$ is $\pi _1$-exact.
\end{lem}

\begin{proof}
We apply \cref{lem:20221020}.
Then for generic $b\in B$, we have $\Delta(p_b)\subset \Delta(p)$, where $p_b:X_b\to A_b$ is the induced algebraic fiber space.
Note that $\overline{r(\Delta (p))}\subsetneqq B$ by $r(\Delta (p))=\Delta(p)/\mathrm{St}^o(\Delta (p))$.
Hence we may assume $b\in B\backslash r(\Delta (p))$.
Then we have $A_b\cap \Delta(p)=\emptyset$.
Hence $\Delta(p_b)=\emptyset$.
We take generic $a\in A$ such that the fiber $X_a$ of $p:X\to A$ over $a$ is smooth.
Then $X_a$ is a smooth fiber of the algebraic fiber space $p_b:X_b\to A_b$.
By \cref{lem:4}, we have the exact sequence:
\begin{equation}\label{eqn:s}
\pi _1(X_a,x)\to \pi _1(X_b,x)\to \pi _1(A_b,a)\to 1,
\end{equation}
where $x\in X_a$.
We consider the following sequence:
\begin{equation}\label{eqn:m}
\begin{CD}
\pi _1(X_a,x) @>>> \pi _1(X,x) @>>> \pi _1(A,a) \\
@V{c}VV @V{i}VV @V{e}VV \\
\pi _1(X_b,x) @>>> \pi _1(X,x) @>>> \pi _1(B,b) 
\end{CD}
\end{equation}
where the second line is exact from the assumption and $i$ is an isomorphism.

Since $B$ is a semi-abelian variety, we have $\pi_2(B)=0$.
Hence the homotopy exact sequence of the fibration
$A_b\to A\to B$ identifies
\[
\pi_1(A_b,a)=\ker\bigl(e:\pi_1(A,a)\to\pi_1(B,b)\bigr).
\]
On the other hand, by \eqref{eqn:s},
\[
\ker\bigl(\pi_1(X_b,x)\to\pi_1(A_b,a)\bigr)
=\operatorname{Im}(c).
\]

Now let $\gamma\in\pi_1(X,x)$ be in the kernel of
$\pi_1(X,x)\to\pi_1(A,a)$.
Since the second line of \eqref{eqn:m} is exact, $\gamma$ is the image
of some $\widetilde{\gamma}\in\pi_1(X_b,x)$.
By the commutativity of \eqref{eqn:m}, the image of
$\widetilde{\gamma}$ in
\[
\pi_1(A_b,a)=\ker(e)
\]
is trivial. Hence $\widetilde{\gamma}\in\operatorname{Im}(c)$.
It follows that $\gamma$ belongs to the image of
$\pi_1(X_a,x)\to\pi_1(X,x)$.
Thus the first line of \eqref{eqn:m} is exact.
\end{proof}

\begin{proof}[Proof of \cref{pro:202210131}]
Inductively, we define algebraic fiber spaces $p_i:X\to A_i$, where $A_i$ are semi-abelian varieties, as follows: First, set $A_1=A$ and $p_1=p$. Next if $\dim \mathrm{St}(\Delta (p_i))>0$, then set $A_{i+1}=A_i/\mathrm{St}^o(\Delta (p_i))$ and let $p_{i+1}$ be the composition of $p_i$ and the quotient map $A_i\to A_{i+1}$. Then since $\dim A_i>\dim A_{i+1}>0$, the induction process   stops as soon as $\Delta (p_i)=\emptyset$. Indeed, when $\Delta(p_i)\neq\varnothing$, by Lemma \ref{lem:3}, we have  $\dim \mathrm{St}(\Delta (p_i))>0$, which shows that $\dim A_{i+1}<\dim A_i$.  Then by   Lemmas \ref{lem:4} and \ref{lem:5}, we conclude that $p$ is $\pi _1$-exact.
\end{proof}

\begin{rem}
	Let us compare our results with those of Javanpeykar  and Rousseau \cite{JR22} on their \emph{geometrically special} projective varieties. By \cite[Theorem 1.10]{JR22}, the Albanese morphism of a normal projective geometrically special variety is surjective, has connected fibers, and has no multiple fibers in codimension one. This gives a natural comparison with \cref{prop:factor}. Moreover, \cite[Theorem 1.15]{JR22} proves that the image of every linear representation of the fundamental group of a smooth projective geometrically special variety is virtually abelian, and \cite[Theorem 1.16]{JR22} shows that such a variety admitting a big representation has a finite \'etale cover birational to an abelian variety. These results are particularly relevant to \cref{main5,char abelian}.
	
	We stress that geometrically special varieties and special or $h$-special varieties are different notions. In particular, the  results in this paper should not be regarded as a formal generalization of those in \cite{JR22}. We also refer to \cite[Remarks 3.26--3.28]{JR22} for the discussion of exceptional divisors related to \cref{lem:3}. In particular, the absence of multiple fibers should not be identified with the triviality of the orbifold base divisor.
\end{rem}
\begin{rem}
	A related result of \cref{lem:20221020} appears in the work by Bartsch, Javanpeykar and Rousseau \cite[Section 2.5, Lemma 2.14]{BJR26}. Under the hypotheses considered there, the orbifold base commutes with restriction to a general fiber:
	$$
	\Delta_{F_t}=(\Delta_F)|_{Y_t}.
	$$
	The discussion in \cite{BJR26} also compares this result with \cref{lem:20221020}. The hypotheses are different: the morphism between the total spaces in \cite{BJR26} is assumed to be generically finite, whereas no such restriction is imposed in Lemma~4.16. Thus this is a comparison with a related statement, rather than a replacement for Lemma~4.16.
\end{rem}
\begin{rem}
	Let us mention an important example in  the work by Bartsch and Javanpeykar \cite[Theorem A]{BJ26}. They prove that
	$$
	U=\mathbb A^2\setminus V(x^2y^3-1)
	$$
	is weakly special but not special. This gives a concrete illustration of why weak specialness cannot replace specialness or $h$-specialness in \cref{lem:3}.  Indeed, consider the morphism $p:U\to\mathbb G_m$ defined by $p(x,y)=1-x^2y^3$, and put $D_x=U\cap V(x)$ and $D_y=U\cap V(y)$. A direct calculation gives
	$$
	p^*[1]=2D_x+3D_y,\qquad \Delta(p)=\frac{1}{2}[1].
	$$
	Thus the inf-multiplicity is two, the gcd-multiplicity is one, and $\operatorname{St}(\Delta(p))^o$ is trivial.
\end{rem}

\subsection{Two Lemmas for the proof of \cref{thm:202210123}}\label{s3}

Before going to give a proof of \cref{thm:202210123}, we prepare two  lemmas. 

\begin{lem}\label{claim:20221014}
Let $\Pi\subset \mathbb C^n$ be a finitely generated additive subgroup which is Zariski dense.
Set $\Sigma=\{ \sigma\in\mathrm{GL}(\mathbb C^n);\ \sigma\Pi= \Pi\}$.
Let $\lambda:\Sigma\to \mathrm{GL}(\Pi\otimes_{\mathbb Z}\overline{\mathbb Q})$ be the induced map.
Let $\widetilde{\Sigma}\subset \Sigma$ be a subgroup and let $E\subset \mathrm{GL}(\Pi\otimes_{\mathbb Z}\overline{\mathbb Q})$ be the Zariski closure of $\lambda(\widetilde{\Sigma})\subset \mathrm{GL}(\Pi\otimes_{\mathbb Z}\overline{\mathbb Q})$.
Let $Y$ be the Zariski closure of $\widetilde{\Sigma}\subset \mathrm{GL}(\mathbb C^n)$.
Assume that the identity component $E^o\subset E$ is unipotent.
Then the identity component $Y^o\subset Y$ is unipotent.  
\end{lem}

\begin{proof}
Since $\Pi\subset \mathbb C^n$ is Zariski dense, $\lambda$ is injective and the natural linear map $q:\Pi\otimes_{\mathbb Z}\mathbb C\to \mathbb C^n$ is surjective.
Set $K=\mathrm{ker}(q)$.
Then we have the following exact sequence of finite dimensional $\mathbb C$-vector spaces:
$$
0\to K\to \Pi\otimes_{\mathbb Z}\mathbb C\to \mathbb C^n\to 0.
$$
Let $H\subset \mathrm{GL}(\Pi\otimes_{\mathbb Z}\mathbb C)$ be the algebraic subgroup such that $g\in H$ if and only if $gK=K$.
We claim that $\lambda(\Sigma)\subset H$ under the embedding $\lambda:\Sigma\hookrightarrow \mathrm{GL}(\Pi\otimes_{\mathbb Z}\mathbb C)$.
Indeed, we take $\sigma\in \Sigma$ and $v_1\otimes c_1+\cdots+v_r\otimes c_r\in K$.
Then we have $c_1v_1+\cdots +c_rv_r=0$ in $\mathbb C^n$.
We have
\begin{equation*}
\begin{split}
q(\lambda(\sigma)(v_1\otimes c_1+\cdots+v_r\otimes c_r))&=q(\sigma(v_1)\otimes c_1+\cdots+\sigma(v_r)\otimes c_r)\\
&=c_1q(\sigma(v_1))+\cdots+c_rq(\sigma(v_r))
\\
&=c_1\sigma(v_1)+\cdots+c_r\sigma(v_r)
\\
&=\sigma(c_1v_1+\cdots+c_rv_r)=0.
\end{split}
\end{equation*}
Hence $\sigma(v_1\otimes c_1+\cdots+v_r\otimes c_r)\in K$.
Thus $\sigma\in H$, so $\Sigma\subset H$.

Now we have a morphism $p:H\to \mathrm{GL}(\mathbb C^n)$.
Then $p\circ\lambda:\Sigma\to \mathrm{GL}(\mathbb C^n)$ is the original embedding.
This induces a surjection $E_{\mathbb C}\to Y$, where $E_{\mathbb C}=E\times_{\bar{\mathbb Q}}\mathbb C$.
Hence we get a surjection $E^o_{\mathbb C}\to Y^o$.
Since every quotient of unipotent group is again unipotent, 
 $Y^o$ is unipotent. 
\end{proof}

\begin{lem}\label{lem:20221015}
Let $Z\to E$ be a surjective morphism of algebraic groups.
Assume that the radical of $Z$ is unipotent.
Then the radical of $E$ is unipotent.
\end{lem}

\begin{proof}
We have an exact sequence
$$
1\to R(Z)\to Z\to Z/R(Z)\to 1,
$$
where $R(Z)$ is the radical of $Z$.
Then $Z/R(Z)$ is semi-simple.
Let $N\subset Z$ be the kernel of $Z\to E$.
Then we have
$$
1\to R(Z)/(R(Z)\cap N)\to E\to (Z/R(Z))/N'\to 1
$$
where $N'\subset Z/R(Z)$ is the image of $N$.
Then $(Z/R(Z))/N'$ is semi-simple and $R(Z)/(R(Z)\cap N)$ is solvable.
Hence $R(E)=R(Z)/(R(Z)\cap N)$.
Hence $R(E)$ is a quotient of $R(Z)$.
Note that $R(Z)$ is unipotent.
Hence $R(E)$ is unipotent.
\end{proof}

\subsection{Proof of \cref{thm:202210123}}\label{s4}
\begin{proof}[Proof of \cref{thm:202210123}]
We take a finite {\'e}tale cover $X'\to X$ such that 
the image of $\pi_1(X')$ in $G/G'$ is torsion free. Since $G$ is connected, it follows that   $\varphi(\pi_1(X'))\subset G$ is Zariski dense and $X'$ is special or $h$-special.
Hence by replacing $X$ by $X'$, we may assume that  image of $\pi_1(X)$ in $G/G'$  is torsion free. We denote by
$$
\varphi_1:\pi_1(X)\to G/G'
$$
the composition of $\varphi$ with the quotient $G\to G/G'$.

Let $X\to A$ be the quasi Albanese map.
We have the following sequence 
\begin{equation}\label{eqn:202210131}
\pi_1(F)\overset{\kappa}{\to} \pi_1(X)\to \pi_1(A)\to 1,
\end{equation}
where $F\subset X$ is a general fiber.
By \cref{pro:202210131}, this sequence is exact.
Since $G/G'$ is commutative, it follows that $\varphi_1$ factors through $\pi_1(X)^{ab}$.  Note that $\pi_1(A)=\pi_1(X)^{ab}/{\rm torsion}$.  Let $\Phi\subset \pi_1(X)$ be the image of $\pi_1(F)\to \pi_1(X)$, and we have a natural morphism
 $
\pi_1(X)'\to \Phi. 
$ It follows that   $
\varphi_1(\Phi)
$ 
is a finite group, and thus trivial,  as we assume that $\varphi_1$ has torsion free image.   Thus we have 
 $\varphi(\Phi)\subset  G' $. 
Let $\Pi\subset G'/G''$ be the image of $\pi_1(F)\to G'/G''$.
Since $\pi_1(F)$ is finitely generated, $\Pi$ is a finitely generated, abelian group.

\begin{claim}\label{claim:202210153}
$\Pi\subset G'/G''$ is Zariski dense.
\end{claim}

\begin{proof}
Set $\Gamma=\varphi(\pi_1(X))$.
Note that $\Gamma' \subset G'$, where $\Gamma'=[\Gamma,\Gamma]$ and $G'=[G,G]$.
We first show that $\Gamma'$ is Zariski dense in $G'$.
Let $H\subset G'$ be the Zariski closure of $\Gamma'$.
Then we have $\Gamma/\Gamma'\to G/H$, whose image is Zariski dense.
Since $\Gamma/\Gamma'$ is commutative, $G/H$ is commutative.
Hence $G'\subset H$.
This shows $H=G'$.
Thus $\Gamma'\subset G'$ is Zariski dense.

By the above arguments,  $\varphi|_{\pi_1(X)'}:\pi_1(X)'\to G'$ factors through $\pi_1(X)'\to \Phi$. This implies that 
$$
\Gamma'\subset \varphi(\Phi). 
$$ 
Hence the image $\varphi(\Phi)\subset  G'$ is Zariski dense, which implies that $\Pi\subset G'/G''$ is Zariski dense.
\end{proof}

By \eqref{eqn:202210131}, we have the following exact sequence:
$$
1\to \Phi\to\pi_1(X)\to \pi_1(A)\to 1.
$$
Note that $\Phi'\subset \pi_1(X)$ is a normal subgroup.
Hence we get
\begin{equation*}
\begin{CD}
1@>>> \Phi^{ab} @>>> \pi _1(X)/\Phi' @>>> \pi _1(A) @>>> 1\\
@. @VVV @VVV @VVV \\
1@>>> G'/G'' @>>> G/G'' @>>> G/G' @>>> 1
\end{CD}
\end{equation*}
By the conjugation, we get 
\begin{equation*}
\pi_1(A)\to \mathrm{Aut}(\Phi^{ab}).
\end{equation*}
This induces
\begin{equation*}
\rho:\pi_1(A)\to \mathrm{Aut}(\Pi).
\end{equation*}
Note that $G'/G''$ is a commutative unipotent group.
Hence $G'/G''\simeq (\mathbb G_a)^n$, where $\mathbb G_a$ is the additive group.
The exponential map $\mathrm{Lie}(G'/G'')\to G'/G''$ is an isomorphism. 
Note that $(\mathbb G_a)^n=\mathbb C^n$ as additive group.
We have 
$$\mathrm{Aut}((\mathbb G_a)^n)=\mathrm{GL}(\mathrm{Lie}(G'/G''))=\mathrm{GL}(\mathbb C^n).
$$
Hence by conjugation, we have 
$$
\mu:G/G'\to \mathrm{Aut}(G'/G'')=\mathrm{GL}(\mathbb C^n).
$$

Let $1\to U\to G\to T\to 1$ be the sequence as in \eqref{eqn:20230321}.
We have $G/G'=(U/G')\times T$, from which we obtain $\mu|_T:T\to\mathrm{GL}(\mathbb C^n) $.
In the following, we are going to prove that $\mu|_T$ is trivial.
(Then \cref{claim:202210121} will yield  that $G$ is nilpotent.)
For this purpose, we shall show that $\mu|_T(T)$ is contained in some unipotent subgroup $Y\subset \mathrm{GL}(\mathbb C^n)$, which we describe below, and then apply \cref{lem:trivial morphism}.

Now we define a subgroup 
$$\Sigma=\{ \sigma\in\mathrm{Aut}(G'/G'');\ \sigma\Pi= \Pi\}\subset \mathrm{GL}(\mathbb C^n).$$
Note that $\rho:\pi_1(A)\to \mathrm{Aut}(\Pi)$ factors through $\mu:G/G'\to \mathrm{Aut}(G'/G'')$.
This induces the following commutative diagram:

\begin{equation}\label{eqn:202210141}
\begin{CD}
\pi _1(A)@>\rho>> \Sigma \\
@VVV @VVV \\
G/G'@>>\mu> \mathrm{GL}(\mathbb C^n)
\end{CD}
\end{equation}

Since $\Pi\subset \bC^n$ is finitely generated, $\Pi$ is a free abelian group of finite rank. 
Since $\Pi\subset \mathbb C^n$ is Zariski dense (cf. \cref{claim:202210153}), the linear subspace spanned by $\Pi$ is $\mathbb C^n$.
Hence we may embed  $\Sigma\subset \mathrm{GL}(\Pi\otimes_{\mathbb Z}\bar{\mathbb Q})$.
Let $E\subset \mathrm{GL}(\Pi\otimes_{\mathbb Z}\bar{\mathbb Q})$ be the Zariski closure of $\rho(\pi_1(A))\subset \mathrm{GL}(\Pi\otimes_{\mathbb Z}\bar{\mathbb Q})$.
Then $E$ is commutative.
Let $E^o\subset E$ be the identity component.

\begin{claim}
$E^o$ is unipotent.
\end{claim}

\begin{proof}
We apply Deligne's theorem.
Let $A^o\subset A$ be a non-empty Zariski open set such that the restriction $f_o:X^o\to A^o$ of $X\to A$ over $A^o$ is a locally trivial $C^\infty$-fibration. 
Then we have the following exact sequence 
$$
\pi_1(F)\to \pi_1(X^o)\to\pi_1(A^o)\to 1.
$$
Let $\Psi\subset \pi_1(X^o)$ be the image of $\pi_1(F)\to \pi_1(X^o)$.
Note that $\Psi'\subset \pi_1(X^o)$ is a normal subgroup.
Then we get the following commutative diagram:

\begin{equation*}
\begin{CD}
1@>>> \Psi^{ab} @>>> \pi _1(X^o)/\Psi' @>>> \pi _1(A^o) @>>> 1\\
@. @VVV @VVV @VVV \\
1@>>>  \Phi^{ab} @>>> \pi _1(X)/\Phi' @>>> \pi _1(A)@>>> 1 \\
@. @VVV @VVV @VVV\\
1@>>> G'/G'' @>>> G/G'' @>>> G/G' @>>> 1
\end{CD}
\end{equation*}
By the conjugation, we get 
\begin{equation*}
\lambda:\pi_1(A^o)\to \mathrm{Aut}(\Psi^{ab}).
\end{equation*} 
This induces
\begin{equation*}
\bar{\lambda}:\pi_1(A^o)\to \mathrm{Aut}(\Pi),
\end{equation*}
which is the composite of $\pi_1(A^o)\to \pi_1(A)$ and $\rho:\pi_1(A)\to  \mathrm{Aut}(\Pi)$.
We note that $\bar{\lambda}$ induces
\begin{equation*}
\bar{\lambda}_{\bar{\mathbb Q}}:\pi_1(A^o)\to\mathrm{GL}(\Pi\otimes_{\mathbb Z}\bar{\mathbb Q}).
\end{equation*}
Then $E\subset \mathrm{GL}(\Pi\otimes_{\mathbb Z}\bar{\mathbb Q})$ is the Zariski closure of the image $\bar{\lambda}_{\bar{\mathbb Q}}(\pi_1(A^o))$.

Now we have the monodromy action
\begin{equation*}
\tau:\pi_1(A^o)\to\mathrm{GL}(\pi_1(F)^{ab}\otimes_{\mathbb Z}\bar{\mathbb Q})
\end{equation*}
induced from the family $X^o\to A^o$.
Let $Z\subset \mathrm{GL}(\pi_1(F)^{ab}\otimes_{\mathbb Z}\bar{\mathbb Q})$ be the Zariski closure of the image $\tau(\pi_1(A^o))$.
Let $R(Z)\subset Z$ be the radical of $Z$. 
Since $X^o\to A^o$ is a locally trivial $C^\infty$ fibration, $\tau$ is dual to the local system $R^1(f_o)_*(\bar{\bQ})$. Note that $R^1(f_o)_*(\bar{\bQ})$, hence $\tau$ underlies an  admissible variation of mixed Hodge structures (cf. \cite{BZ14} for the definition). 
Then by Deligne's theorem (cf. \cite[Corollary 2]{And92} or \cite[4.2.9b]{Del71}), $R(Z)$ is unipotent. 
Let $H\subset \mathrm{GL}(\pi_1(F)^{ab}\otimes_{\mathbb Z}\bar{\mathbb Q})$ be an algebraic subgroup defined by
$$
H=\{\sigma \in \mathrm{GL}(\pi_1(F)^{ab}\otimes_{\mathbb Z}\bar{\mathbb Q}); \sigma K= K\},
$$
where $K=\mathrm{ker}(\pi_1(F)^{ab}\otimes_{\mathbb Z}\bar{\mathbb Q}\to \Pi\otimes_{\mathbb Z}\bar{\mathbb Q})$.
Then $\tau(\pi_1(A^o))\subset H$, hence $Z\subset H$.
Note that $\bar{\lambda}_{\bar{\mathbb Q}}$ factors through $\tau:\pi_1(A^o)\to H$. 
This induces a morphism $Z\to E$, which is
surjective, for $\bar{\lambda}_{\bar{\mathbb Q}}(\pi_1(A^o))$ is Zariski dense.
Hence we get a surjection $Z\to E$.
Hence by \cref{lem:20221015}, $R(E)$ is unipotent.
Since $E$ is commutative, we have $R(E)= E^o$.
Hence $E^o$ is unipotent.
\end{proof}

Let $Y\subset \mathrm{GL}(\mathbb C^n)$ be the Zariski closure of $
\rho(\pi_1(A))\subset \mathrm{GL}(\mathbb C^n)$.
Let $Y^o\subset Y$ be the identity component.
By \cref{claim:202210153}, we may apply \cref{claim:20221014} for 
$\widetilde{\Sigma}=\rho(\pi_1(A))\subset\Sigma$ to conclude that $Y^o$ is unipotent.
Since the image of $\pi_1(A)\to G/G'$ is Zariski dense, the commutativity of \eqref{eqn:202210141} implies $\mu(G/G')\subset Y$.
This is Zariski dense, in particular $Y^o=Y$.
By $G/G'= (U/G')\times T$, $\mu$ induces $\mu|_T:T\to Y$.
Since $Y$ is unipotent, this is trivial by \cref{lem:trivial morphism}. 
Hence the action of $T$ onto $G'/G''$ is trivial.
By Lemma \ref{claim:202210121}, $G$ decomposes as a direct product $G \cong U \times T$, and $G$ is nilpotent.
\end{proof}

\subsection{Proof of \cref{thm:VN}}  \label{s5}
\cref{thm:VN} is a consequence of \cref{cor:202304071,thm:202210123}. 
\begin{proof}[Proof of \cref{thm:VN}] 
Let $G$ be the Zariski closure of $\varrho$.
Let $G_0$ be the identity component of $G$, i.e., the connected component which contains the identity element.
Let $X'\to X$ be a finite \'etale cover   corresponding to the finite index subgroup $\varrho^{-1}(\varrho(\pi_1(X))\cap G_0(\bC))$ of $\pi_1(X)$.
The property of being special or $h$-special is preserved from $X$ to $X'$ (cf. \cref{special,lem:202304061}).
The restriction $\varrho|_{\pi_1(X')}:\pi_1(X')\to G_0(\bC)$ has Zariski dense image.
 Denote by $R(G_0)$   the radical of $G_0$, which is the unique maximal connected solvable normal algebraic subgroup.
 If $R(G_0)\neq G_0$, the quotient $G_0/R(G_0)$ is semisimple.
 Note that the  induced representation $\varrho':\pi_1(X')\to G_0/R(G_0)(\bC)$ is still Zariski dense.   By \cref{cor:202304071}, $X'$ is neither $h$-special nor  special which contradicts our assumption. This implies that $R(G_0)=G_0$.     
Therefore by \cref{thm:202210123}, $G_0$  decomposes as a desired direct product $G_0 \cong U \times T$, where $U$ is unipotent, and $T$ is an algebraic torus.
 In particular, $G_0$ is nilpotent.
 Hence $\varrho(\pi_1(X))$ is virtually nilpotent.

	If  $\varrho$ is assumed to be reductive, then $G$ is reductive. Hence
	$G_0$ is an algebraic torus.  It follows that $\varrho(\pi_1(X))$ is a virtually abelian group. 
\end{proof}
\subsection{Illustrative examples}\label{sec:examples} 
As mentioned in the introduction, the following example disproves \cref{conj:Campana}. 
We construct a quasi-projective surface that is both $h$-special and special, whose fundamental group is linear and nilpotent but not virtually abelian. 
 \begin{example}\label{example} 
	Fix $\tau\in\mathbb H$ from the upper half plane.
	Then $\mathbb C/(\mathbb Z+\mathbb Z\tau)$ is an elliptic curve.
	We define a nilpotent group $G$ as follows.
	$$
	G=\left\{g(l,m,n)=\begin{pmatrix}
		1 & 0 & m & n \\
		-m & 1 & -\frac{m^2}{2} & l \\
		0 & 0 & 1 & 0 \\
		0 & 0 & 0 & 1
	\end{pmatrix}
	\in\mathrm{GL}_4(\bQ)
	\
	\vert 
	\
	l,m,n\in \mathbb Z
	\right\}
	$$
	Thus as sets $G\simeq \mathbb Z^3$.
	However, $G$ is non-commutative as direct computation shows:
	\begin{equation}\label{eqn:202210061}
		g(l,m,n)\cdot g(l',m',n')=g(-mn'+l+l',m+m',n+n').
	\end{equation}
	We define $C\subset G$ by letting $m=0$ and $n=0$.
	$$
	C=\left\{
	g(l,0,0)=
	\begin{pmatrix}
		1 & 0 & 0 & 0 \\
		0 & 1 & 0 & l \\
		0 & 0 & 1 & 0\\
		0 & 0 & 0 & 1
	\end{pmatrix}
	\in\mathrm{GL}_4(\bQ)
	\
	\vert 
	\
	l\in \mathbb Z
	\right\}
	$$
	Then $C$ is a free abelian group of rank one, thus $C\simeq \mathbb Z$ as groups.
	By \eqref{eqn:202210061}, $C$ is the  center of $G$.  
	We have an exact sequence
	$$
	1\to C\to G\to L\to 1,
	$$
	where $L\simeq \mathbb Z^2$ is a free abelian group of rank two.
	This is a central extension.
	The quotient map $G\to L$ is defined by $g(l,m,n)\mapsto (m,n)$.
	
	\begin{claim}
		$G$ is not almost abelian.
	\end{claim}
	
	\begin{proof}
		For $(\mu,\nu)\in\mathbb Z^2$, we set $G_{\mu,\nu}=\{ g(l,m,n);\ \mu n=\nu m\}$.
		Then by \eqref{eqn:202210061}, $G_{\mu,\nu}$ is a subgroup of $G$.
		By \eqref{eqn:202210061}, $g(l,m,n)$ commutes with $g(l',m',n')$ if and only if $g(l',m',n')\in G_{m,n}$.
		We have $C\subset G_{\mu,\nu}$ and $G_{\mu,\nu}/C=\{ (m,n)\in L; \mu n=\nu m\}$.
		Hence the index of $G_{\mu,\nu}\subset G$ is infinite for $(\mu,\nu)\not=(0,0)$.

		Now assume contrary that $G$ is almost abelian.
		Then we may take a finite index subgroup $H\subset G$ which is abelian.
		We may take $g(l,m,n)\in H$ such that $(m,n)\not=(0,0)$.
		Then we have $H\subset G_{m,n}$.
		This is a contradiction, since the index of $G_{m,n}\subset G$ is infinite.
	\end{proof}

	Now we define an action $G\curvearrowright \mathbb C^2$ as follows:
	For $(z,w)\in\mathbb C^2$, we set
	$$
	\begin{pmatrix}
		z  \\
		w\\
		\tau \\
		1 
	\end{pmatrix}
	\mapsto
	\begin{pmatrix}
		1 & 0 & m & n \\
		-m & 1 & -\frac{m^2}{2} & l \\
		0 & 0 & 1 & 0 \\
		0 & 0 & 0 & 1\end{pmatrix}
	\begin{pmatrix}
		z  \\
		w\\
		\tau  \\
		1 
	\end{pmatrix}
	$$
	Hence
	$$
	g(l,m,n)\cdot (z,w)=\left(z+m\tau+n,-mz+w-\frac{m^2}{2}\tau+l\right).
	$$
	This action is properly discontinuous.
	We set $X=\mathbb C^2/G$.
	Hence 
	$\pi_1(X)=G.$ 
	Then $X$ is a smooth complex manifold.
	We have $\mathbb C^2/C\simeq \mathbb C\times \mathbb C^*$.
	The action $L\curvearrowright \mathbb C\times \mathbb C^*$ is written as 
	\begin{equation}\label{eqn:20221005}
		(z,\xi)\mapsto (z+m\tau+n,e^{-2\pi i mz-\pi im^2\tau}\xi),
	\end{equation}
	where $\xi=e^{2\pi iw}$.
	The first projection $\mathbb C\times \mathbb C^*\to \mathbb C$ is equivariant $L\curvearrowright \mathbb C\times \mathbb C^*\to \mathbb C\curvearrowleft L$.
	By this, we have $X\to E$, where $E=\mathbb C/<\mathbb Z+\mathbb Z\tau>$ is an elliptic curve.
	The action \eqref{eqn:20221005} gives the action on $\mathbb C\times \mathbb C$ by the natural inclusion $\mathbb C\times \mathbb C^*\subset \mathbb C\times \mathbb C$.
	We consider this as a trivial line bundle by the first projection $\mathbb C\times \mathbb C\to \mathbb C$. 
	We set $Y=(\mathbb C\times \mathbb C)/L$.
	This gives a holomorphic line bundle $Y\to E$. 
	By Serre's GAGA, $Y$ is algebraic.
	Hence $X=Y-Z$ is quasi-projective, where $Z$ is the zero section of $Y$.

	\begin{claim}\label{claim:hspecial}
The quasi-projective surface $X$ has vanishing logarithmic Kodaira dimension $\bar{\kappa}(X)=0$ and contains a metrically dense entire curve. In particular, it is special and $h$-special.
\end{claim}

\begin{proof}
We first show $\bar{\kappa}(X)=0$.
Let $\mathcal L$ be the invertible sheaf associated with the line bundle
$Y\to E$.
We consider the natural compactification
\[
\overline X=\mathbb P(\mathcal{L}^*\oplus\mathcal O_E).
\]
We write $\xi:=\cO_{\overline{X}}(1)$ for the tautological line bundle. 
Denote by  $q:\overline{X}\to E$   the projection map.  
Then   $\cO_{\overline{X}}(Z)=q^*\mathcal{L}+ \xi$, where $Z\subset Y$ is the zero section.
Denote by $D:=\overline{X}\backslash Y$. Then   $\cO_{\overline{X}}(D)=\xi$.   
Note that $K_{\overline{X}}=-2\xi+q^*(K_E+\det (\mathcal{L}^*\oplus \cO_E))$. 	It follows that 
\begin{equation}\label{eqn:20260925}
K_{\overline{X}}+Z+D=\cO_{\overline{X}}.
\end{equation} 
This shows $\bar\kappa(X)=0$.
Thus $X$ is special by \cite[Théorème 7.7]{Cam11}.

We next show that $X$ contains a metrically dense entire curve.
We take a dense set $\{(x_n,y_n)\}_{n\in \bZ}\subset  \bC^2$. By  Mittag-Leffler interpolation we can find entire functions $f_1(z)$ and $f_2(z)$   such that $f_1(n)=x_n$ and   $f_2(n)=y_n$. Then $f:\bC\to \bC^2$ defined by $f=(f_1,f_2)$ is a metrically dense entire curve. 
	Note that $\pi:\bC^2\to X$ is the universal covering map. It follows that $\pi\circ f$  is  a metrically dense entire curve in  $X$. Hence $X$ is also $h$-special. 
\end{proof}
\end{example}

\begin{rem}\label{re:20230428}
Let $Y\to E$ be the line bundle described in \Cref{example}.
Here $Y=(\mathbb C\times \mathbb C)/L$ under the action given by \eqref{eqn:20221005}.
By the $\tau$-quasiperiodic relation, the Jacobi theta function $\vartheta(z,\tau)$ gives an equivariant section of the first projection $\mathbb C\times \mathbb C\to \mathbb C$.
		Hence the degree of the line bundle $Y\to E$ is equal to one.
Thus $Y$ is ample.
\end{rem}

\begin{rem}
The above $X$ in \Cref{example} is homotopy equivalent to a Heisenberg manifold which is a circle bundle over the 2-torus.
It is well known that Heisenberg manifolds have nilpotent fundamental groups.
\end{rem}

%\begin{example}
%	This example is $h$-special complex manifold with non virtually abelian fundamental group.

%	Let $E=\bC/\Lambda$ be an elliptic curve, where $\Lambda$ is a lattice generated by $1$ and $\tau$. Define a equivalent relation $\bC^*\times\bC^*\times \bC$ defined by  
%	$$
%	(\xi_1,\xi_2,z)\sim (\xi_1\xi^m_2,\xi_2,z+m+n\tau)
%	$$
%	The quotient is a holomorphic fiber bundle $X$ over $E$  with fibers $\bC^*\times\bC^*$. For the monodromy representation of this fiber bundle  $\varrho:\bZ^2\to Aut(\bZ^2)$,  $\varrho((1,0))$ is the identity map and  
%	\begin{equation*}
%	\varrho((0,1))=	\begin{bmatrix}
%		1 & 1\\
%		0 & 1
%	\end{bmatrix}
%	\end{equation*} 
%	It seems that $\pi_1(X)$ cannot be virtually abelian since the monodromy action distorts the fibers each time. 

%\end{example}

We will construct an example of a complex manifold admitting a
metrically dense entire curve whose fundamental group is linear and
solvable, but not virtually nilpotent.
It cannot be quasi-projective: otherwise, by \cref{lem:20230406},
it would be $h$-special, contradicting \cref{thm:202210123}.

\begin{example} 
	We start from algebraic argument.
	Set
	\begin{equation*}
		M=	\begin{bmatrix}
			1 & 2\\
			1 & 3
		\end{bmatrix}
		\in\mathrm{SL}_2(\mathbb Z)
	\end{equation*} 	
	All the needed property for $M$ is contained in the following \cref{claim:202210091}.
	
	\begin{claim}\label{claim:202210091}
		For every non-zero $n\in \mathbb Z-\{0\}$, $M^n$ has no non-zero eigenvector in $\mathbb Z^2$.
	\end{claim}
	
	\begin{proof}
		Fix $n\in \mathbb Z-\{0\}$.
		Assume contrary that there exists $v\in \mathbb Z^2\backslash \{ 0\}$, such that $M^nv=\alpha^nv$, where $\alpha$ is an eigenvalue of $M$.
		Then $\alpha^n\not\in\mathbb Q$ as direct computation shows. 
		Hence $\alpha^nv\not\in \mathbb Z^2$, while $M^nv\in\mathbb Z^2$.
		This is a contradiction.
	\end{proof}
	
	\begin{claim}\label{claim:202210092}
		Let $N\subset \mathbb Z^2$ be a submodule.
		Let $n\in \mathbb Z-\{0\}$.
		If $N$ is invariant under the action of $M^n$, then either $N=\{0\}$ or $N$ is finite index in $\mathbb Z^2$.
	\end{claim}
	
	\begin{proof}
		Suppose $N\not=\{0\}$.
		By \cref{claim:202210091}, we have $\dim_{\mathbb R}N\otimes_{\mathbb Z}\mathbb R=2$.
		Then $N$ is finite index in $\mathbb Z^2$.
	\end{proof}
	
	\begin{claim}\label{claim:202210093}
		Let $G$ be a group which is an extension 
		$$
		1\to \mathbb Z^2\to G\overset{p}{\to}\mathbb Z\to 1.
		$$
		Assume that the corresponding action $\mathbb Z\to {\rm Aut}(\bZ^2)$ is given by $n\mapsto M^n$.
		Then $G$ is not virtually nilpotent.
	\end{claim}
	
	\begin{proof}
		Assume contrary that there exists a finite index subgroup $G'\subset G$ which is nilpotent.
		Then there exists a central sequence 
		$$G'\cap \mathbb Z^2=N_0\supset N_1\supset N_2\supset \cdots\supset N_{k+1}=\{0\},$$ 
		i.e., for each $i=0,\ldots, k$, $N_i\subset G'$ is a normal subgroup and $N_i/N_{i+1}\subset G'/N_{i+1}$ is contained in the center of $G'/N_{i+1}$.
		By induction on $i$, we shall prove that $N_i\subset \mathbb Z^2$ is a finite index subgroup.
		For $i=0$, this is true by $N_0=G'\cap \mathbb Z^2$.
		We assume that $N_i\subset \mathbb Z^2$ is finite index.
		There exists $l\in\mathbb Z_{>0}$ such that $p(G')=l\mathbb Z$, where $p:G\to\mathbb Z$.
		Then we have an exact sequence
		$$
		1\to G'\cap \mathbb Z^2\to G'\to l\mathbb Z\to 1.
		$$
		Since $N_{i+1}\subset G'$ is normal, $N_{i+1}\subset \mathbb Z^2$ is invariant under the action $M^l$.
		By \cref{claim:202210092}, either $N_{i+1}=\{0\}$ or $N_{i+1}\subset \mathbb Z^2$ is finite index.
		In the second case, we complete the induction step.
		Hence it is enough to show $N_{i+1}\not=\{0\}$.
		So suppose $N_{i+1}=\{0\}$.
		Note that $N_i/N_{i+1}$ is contained in the center of $G'/N_{i+1}$.
		Hence $N_i$ is contained in the center of $G'$.
		So $M^l$ acts trivially on $N_i$.
		This is a contradiction (cf. \cref{claim:202210091}).
		Hence $N_{i+1}\not=\{0\}$.
		Thus we have proved that $N_i\subset \mathbb Z^2$ is a finite index subgroup.
		This contradicts to $N_{k+1}=\{0\}$.
		Thus $G$ is not virtually nilpotent.
	\end{proof}
	
	Now for $n\in\mathbb Z$, we define integers $a_n$, $b_n$, $c_n$ ,$d_n$ by
	\begin{equation*}
		\begin{bmatrix}
			a_n & b_n\\
			c_n & d_n
		\end{bmatrix}
		=M^n
	\end{equation*} 	
	Define a $\bZ$-action on $\bC^*\times\bC^*\times \bC$ by  
	$$
	(\xi_1,\xi_2,z)\mapsto (\xi_1^{a_n}\xi^{b_n}_2,\xi_1^{c_n}\xi^{d_n}_2,z+n).
	$$
	The quotient by this action is a holomorphic fiber bundle $X$ over $\mathbb C^*$  with fibers $\bC^*\times\bC^*$. For the monodromy representation of this fiber bundle  $\varrho:\bZ\to \mathrm{Aut}(\bZ^2)$,  
	\begin{equation*}
		\varrho(n)=M^n
	\end{equation*} 
	We have
	$$
	1\to \pi_1(\mathbb C^*\times \mathbb C^*)\to \pi_1(X)\to \pi_1(\mathbb C^*)\to 1
	$$
	Hence $\pi_1(X)$ is solvable.
	By \cref{claim:202210093}, $\pi_1(X)$ is not virtually nilpotent. 
	
	We claim that $\pi_1(X)$ is linear. Indeed, since the extension by $\bZ$ splits and we have
	$$
	\pi_1(X)\simeq \bZ^2\rtimes_{M}\bZ, 
	$$ it follows that the representation
	$$
	(v,n)\mapsto \begin{bmatrix}
		M^n & v\\
		0 & 1 
	\end{bmatrix}\in {\rm GL}_3(\bZ)
	$$
	is faithful. 
	
	As the universal cover of $X$ is $\bC^3$, by the same arguments in \cref{claim:hspecial}, $X$ admits a metrically dense entire curve.  
\end{example}

\begin{example}\label{example:20221105}
	This example is a quasi-projective surface $X$ with maximal quasi-Albanese dimension such that $X$ is $h$-special, special, and $\bar{\kappa}(X)=1$.
	
	Let $C_1$ and $C_2$ be elliptic curves which are not isogenous.
	Set $A=C_1\times C_2$.
	Let $p_1:A\to C_1$ and $p_2:A\to C_2$ be the first and second projections, respectively.
	Let $\widetilde{A}=\mathrm{Bl}_{(0,0)}A$ be the blow-up with respect to the point $(0,0)\in A$.
	Let $E\subset \widetilde{A}$ be the exceptional divisor and let $D\subset \widetilde{A}$ be the proper transform of the divisor $p_1^{-1}(0)\subset A$.
	Set $X=\widetilde{A}-D$.
	Then by \cref{lem:same Kd}, we have $\bar{\kappa}(X)=\bar{\kappa}(X-E)$.
	On the other hand, we have $X-E=(C_1-\{0\})\times C_2$.
	Hence $\bar{\kappa}(X)=1$.
	
	Next we construct a Zariski dense entire curve $f:\mathbb C\to X$.
	Let $\pi_1:\mathbb C\to C_1$ and $\pi_2:\mathbb C\to C_2$ be the universal covering maps.
	We assume that $\pi_1(0)=0$ and $\pi_2(0)=0$.
	Set $\Gamma=\pi_1^{-1}(0)$, which is a lattice in $\mathbb C$.
	We define an entire function $h(z)$ by the Weierstrass canonical product as follows
$$
	h(z) = z\prod_{\omega \in \Gamma \setminus \{0\}}
	\left( 1 - \frac{z}{\omega} \right)
	\exp\left( \frac{z}{\omega} + \frac{z^2}{2\omega^2} \right).
$$
	We consider $\pi_2\circ h:\mathbb C\to C_2$.
	Then for $\omega\in \Gamma$, we have $\pi_2\circ h(\omega)=0$.
	Hence $\Gamma\subset (\pi_2\circ h)^{-1}(0)$.
	Thus $\pi_1^{-1}(0)\subset (\pi_2\circ h)^{-1}(0)$.
	We define $g:\mathbb C\to A$ by $g(z)=(\pi_1(z),\pi_2\circ h(z))$.
	Then $g^{-1}(p_1^{-1}(0))\subset g^{-1}((0,0))$.
	Let $f:\mathbb C\to \widetilde{A}$ be the map induced from $g$.
	Then $f^{-1}(D+E)\subset f^{-1}(E)$.
	By $\pi_1'(z)\not=0$ for all $z\in \mathbb C$, we have $f^{-1}(D\cap E)=\emptyset$.
	Hence $f^{-1}(D)=\emptyset$.
	This yields $f:\mathbb C\to X$.
	By the Bloch-Ochiai theorem, the Zariski closure of $g:\mathbb C\to A$ is a translate of an abelian subvariety of $A$.
	Since $C_1$ and $C_2$ are not isogenous, non-trivial abelian subvarieties are $A$, $C_1\times\{0\}$ and $\{0\}\times C_2$.
	Hence $f$ is Zariski dense. 
	Hence $X$ is $h$-special (cf. \cref{lem:20230406}).
	
	Next we show that $X$ is special. We first note that $X$ is not of log general type as $\bar{\kappa}(X)=1$. 
	If $X$ is not special, then  by definition after taking a proper birational modification $\mu:X'\to X$, there is an algebraic fiber space $g:X'\to Y$ whose orbifold base $(Y,\Delta)$ of $g$ defined in \cref{sec:spechspecial} is of log general type. Since $\dim X=2$ and $\bar{\kappa}(X)=1$, we must have
$\dim Y=1$. Thus $Y=C$ is a quasi-projective curve and the orbifold base $(C,\Delta)$ of $g$ is hyperbolic.

	Since $f:\mathbb C\to X$ is Zariski dense, there exists a lift $\varphi:\mathbb C\to X'$ so that $\mu\circ\varphi=f$.
	Since $\varphi$ is Zariski dense, $g\circ\varphi:\mathbb C\to C$ is non-constant.
	However, the composition $g\circ \varphi$ is an orbifold entire curve of  $(C,\Delta)$. This contradicts the hyperbolicity of $(C,\Delta)$.  Therefore, $X$ is special. \end{example}
\begin{rem}
Since the submission of the initial version of this paper to the arXiv, Aguilar and Campana have  obtained a more general construction of special quasi-projective surfaces whose  fundamental groups are nilpotent but not virtually abelian.  Specifically, they consider the complement of the zero section of a holomorphic line bundle $L$ over an elliptic curve $B$, where the first Chern class $c_1(L) \neq 0$. We refer the interested reader   to \cite[Example 24]{AC25} for more details. 
\end{rem}
\section[A structure theorem]{A structure theorem}      
In this section, we prove \cref{thm:20230510}.  
\subsection{A structure theorem}
Before going to prove this, we prepare the following generalization of a structure theorem for quasi-projective varieties of maximal quasi-albanese dimension in \cref{lem:abelian pi0}. 
For a quasi-projective variety $Y$, we set
$$\Spalg(Y) := \overline{\bigcup_{V} V}^{\mathrm{Zar}},$$
where $V$ ranges over all positive-dimensional closed subvarieties of $Y$ which are not of log general type.

\begin{thm}\label{thm:202305101} 
Let $X$ be a smooth quasi-projective variety and let $g:X\to \cA\times Y$ be a morphism, where $\cA$ is a semi-abelian variety and $Y$ is a smooth quasi-projective variety such that $\Spalg(Y)\subsetneqq Y$.
Let $p:X\to Y$ be the composition of $g:X\to \cA\times Y$ and the second projection $\cA\times Y\to Y$.
Assume that $p:X\to Y$ is dominant and $\dim g(X)=\dim X$.
Then after replacing $X$ by a finite \'etale cover and a proper birational modification, there are a semiabelian variety $A$, a smooth quasi-projective variety $V$  and a birational morphism $a:X\to V$  such that the following commutative diagram holds:
		\begin{equation*} 
			\begin{tikzcd}
				X \arrow[rr,    "a"] \arrow[dr, "j"] & & V \arrow[ld, "h"]\\
				& J(X)&
			\end{tikzcd}
		\end{equation*}
		where $j$ is the logarithmic Iitaka fibration of $X$ and $h:V\to J(X)$ is a  locally analytically trivial fibration with fibers isomorphic to $A$.  
Moreover, for a   general fiber $F$ of $j$, $a|_{F}:F\to A$ is proper in codimension one.  
\end{thm}

\begin{proof}
	Let us first show that $\bar{\kappa}(X)\geq 0$.  Denote by $h:X\to \cA$ the composition of $g$ and the first projection $\cA\times Y\to \cA$.  Since $p:X\to Y$ is dominant and $\dim g(X)=\dim X$, for a general fiber $F$ of $p$, we have $\dim h(F)=\dim F$. Therefore, $\bar{\kappa}(F)\geq 0$. Since we assume that ${\rm Sp}_{\rm alg}(Y)\subsetneq Y$, $Y$ is thus of log general type. By \cite[Theorem 1.9]{Fuj17},  it follows that 
	$$\bar{\kappa}(X)\geq \bar{\kappa}(Y)+\bar{\kappa}(F)\geq 0.$$

	Consider the logarithmic Iitaka fibration
$j:X\dashrightarrow J(X)$.
After replacing the base by a smooth birational model and $X$ by a suitable
proper birational modification, we may assume that $J(X)$ is smooth and
that $j:X\to J(X)$ is a morphism.
Write $X_t:=j^{-1}(t)$.

\begin{claim}\label{claim:20230518}
 $p(X_t)$ is a point for very generic $t\in J(X)$.
 \end{claim}
 
 \begin{proof}
Since $p:X\to Y$ is dominant and $\Spalg(Y)\subsetneqq Y$, we have $\overline{p(X_t)}\not\subset \Spalg(Y)$ for generic $t\in J(X)$.
Hence, $\overline{p(X_t)}$ is of log general type for generic $t\in J(X)$.
Now we take very generic $t\in J(X)$.
To show that $p(X_t)$ is a point, we assume contrary that $\dim p(X_t)>0$.
Then $\bar{\kappa}(\overline{p(X_t)})>0$.
Since $\bar{\kappa}(X_t)=0$, general fibers of $p|_{X_t}:X_t\to \overline{p(X_t)}$ has  non-negative logarithmic Kodaira dimension.  By \cite[Theorem 1.9]{Fuj17} again,  it follows that $\bar{\kappa}(X_t)\geq \bar{\kappa}(\overline{p(X_t)})>0$. We obtain the  contradiction. 
Thus $p(X_t)$ is a point. 
\end{proof}

Let $\alpha:X\to \cA$ be the composition of $g:X\to \cA\times Y$ and the first projection $\cA\times Y\to \cA$.
	Since $\dim X=\dim g(X)$,  one has $\dim X_t=\dim \alpha(X_t)$ for general $t\in J(X)$ by \cref{claim:20230518}.		
	For very generic $t\in J(X)$, note that $\bar{\kappa}(X_t)=0$, hence by \cref{prop:Koddimabb}, the closure of $\alpha(X_t)$ is a translate of a  semi-abelian variety $A_t$  of $\cA$.  
Let $\Lambda \subset J(X)$ denote the set of all such points $t$ that additionally satisfy \cref{claim:20230518}.
	Note that $\cA$ has only at most countably many  semi-abelian subvarieties, it follows that $A_t$ does not depend on very general $t$ which we denote by $B$.

Indeed, we may take $B$ so that the set of $t \in J(X)$ satisfying $A_t = B$ has positive measure; we denote this set by $\Lambda' \subset \Lambda$.
	Let $q:X\to (\cA/B)\times J(X)$ be the natural map.
	Since $q(X)$ is a constructible subset of $(\mathcal{A}/B) \times J(X)$, we may take a locally closed subset $Z \subset (\mathcal{A}/B) \times J(X)$ that is contained in $q(X)$ as a dense subset.
	If $t \in \Lambda'$, then the fiber $\xi^{-1}(t)$ of the natural map $\xi: Z \to J(X)$ consists of a single point.
	Hence the fact that $\Lambda'$ has positive measure implies the existence of a dense Zariski open set $U \subset J(X)$ over which $\xi: Z \to J(X)$ is bijective.
	This implies that $A_t\subset B$ for every $t\in U\cap\Lambda$.
Next, consider the map $q: X \to (\mathcal{A}/B) \times J(X)$. 
Over the subset $\xi^{-1}(U \cap \Lambda')$, the fibers of $q$ have dimension $\dim B$.
Since $\xi^{-1}(U \cap \Lambda')$ has positive measure, it follows that the general fibers of $q: X \to (\mathcal{A}/B) \times J(X)$ over $Z$ have dimension $\dim B$.
Hence, for a general $t \in J(X)$ that belongs to $\Lambda$, we have $B \subset A_t$.
Therefore we have $B=A_t$ for very general $t\in J(X)$ as desired.

	\begin{claim}\label{claim:20221105}
		There are 
		\begin{equation*}
			\begin{tikzcd}
				X' \arrow[r] \arrow[d, "\gamma"] & X\arrow[d,"\alpha"]\\
				\cA' \arrow[r] & \cA
			\end{tikzcd}
		\end{equation*}
		where the two rows are a finite     \'etale cover $X'\to X$ and an isogeny $\cA'\to \cA$  such that for a very general fiber $F$ of the logarithmic Iitaka fibration $j':X'\to J(X')$ of $X'$,  $\gamma|_{F}:F\to \cA'$ is mapped birationally to a translate of a semiabelian subvariety $C$ of $\cA'$, and the induced map $F\to C$ is proper in codimension one.
	\end{claim} 
	\begin{proof} 
	Choose a Zariski open subset $U\subset J(X)$ over which $j$ is topologically locally trivial. Let $X_U:=j^{-1}(U)$ and $j_U=j|_{X_U}:X_U\to U$ is a topologically locally trivial fibration. Then we have
	\begin{equation}\label{eq:diaU}
		\begin{tikzcd}
			 X_U\arrow[rr]\arrow[dr,"j_U"] && \cA\times U\arrow[ld]\\
		&	 U&
		\end{tikzcd}
	\end{equation} For $t\in U$, write $X_t:=j^{-1}(t)$ and  set $\Lambda_t:={\rm Im}[H_1(X_t,\mathbb Z)\to H_1(\cA,\mathbb Z)]$. The groups $H_1(X_t,\mathbb Z)$ form a local system on $U$, and the global morphism $\alpha:X\to\cA$ induces a morphism from this local system to the constant local system on $U$ with fiber $H_1(\cA,\mathbb Z)$. It follows that the subgroups $\Lambda_t\subset H_1(\cA,\mathbb Z)$   are independent of $t$. Since $\pi_1(\cA)$ is abelian, we  have
	$$
	 \Lambda_t:={\rm Im}[\pi_1(X_t,\mathbb Z)\to \pi_1(\cA,\mathbb Z)] 
	$$
	that is independent of $t$. 
	
		For very general $t\in U$, the morphism $\alpha|_{X_t}:X_t\to\cA$ factors 
		$$
		X_t\longrightarrow C_t\longrightarrow B_t\longrightarrow\cA,
		$$
		where $X_t\to C_t$ is birational and proper in codimension one,  $C_t\to B_t$ is an isogeny, and $B_t$ is a translate of $B$. By \cref{lem:abelian pi0}, the morphism $X_t\to C_t$ is birational and proper in codimension one. In particular, the natural morphism $\pi_1(X_t)\to\pi_1(C_t)$ is isomorphic.  
		Since $\pi_1(B)\to\pi_1(\cA)$ is injective and $\Lambda_t$ is independent of $t$,  the isogenies $C_t\to B_t$ are all isomorphic for very general $t$. Denote this common isogeny by $C\to B$.  
		
Now we apply \cref{lem:finite cover} below and let $\cA'\to \cA$ be the isogeny therein such that $C\times_{\cA}\cA'$ is a disjoint union of $C$ and the natural morphism of $C\to \cA'$ is injective.  
 It follows that  for a connected component $X'$ of $X\times_A\cA'$, the fiber of $X'\to J(X)$ at $t$ is a disjoint union of $X_t$.  
 Consider  the quasi-Stein factorisation of $X'\to J(X)$  and we obtain an algebraic fiber space $X'\to I$ and a finite morphism $\beta:I\to J(X)$.   
 Then for each $t'\in I$ with $t=\beta(t')$, the fiber $X'_{t'}$ of $X'\to I$ is isomorphic to $X_t$ hence it has zero logarithmic Kodaira dimension.   
 By the universal property of logarithmic Iitaka fibration,   $X'_{t'}$  is contracted by the logarithmic Iitaka fibration of $X'\to J(X')$. 
 Since $\bar{\kappa}(X')=\bar{\kappa}(X)$ by \cite[Lemma 4.8]{CDY22}, it follows that $X'\to I$ is the logarithmic Iitaka fibration of $X'$.   Moreover, by our construction, for the natural morphism $\gamma:X'\to \cA'$,  the induced map $\gamma|_{X'_{t'}}:X'_{t'}\to \cA'$ is mapped birationally to a translate of $C\subset \cA'$, and proper in codimension one.
 \end{proof}

We replace $X$ and $\cA$ by the above \'etale covers.   
Write $W$ to be the closure of $\alpha(X)$. Then $W$ is invariant under the action of $C$ by the translate.  Denote by $T:=W/C$. The natural morphism $X\to T$ contracts general fibers of $j$. Therefore after replacing by a birational model of $X\to J(X)$, one has
	\begin{equation*}
		\begin{tikzcd}
			X \arrow[r] \arrow[d] &  W\arrow[d] \\
			J(X)  \arrow[r] & T
		\end{tikzcd}
	\end{equation*} 
	Consider the natural morphism $a: X\to J(X)\times_T W$. Then it is birational by the above claim.   Moreover,  for a very general fiber $X_t$, the morphism $a|_{X_t}:X_t\to (J(X)\times_T W)_t$ 
	 is proper in codimension one by \cref{claim:20221105}.  
	 	 
Set $V:=J(X)\times_T W$. 	 Let $\overline{X}$ be a partial compactification of $X$ such that $a:X\to V$ extends to  a proper morphism $\bar{a}:\overline{X}\to V$.   Then $\Xi:=\bar{a}(\overline{X}\backslash X)$ is a Zariski closed subset of $V$.   By the above result  for a very general fiber $V_t$ of the fibration $V\to J(X)$  we know that $V_t\cap \Xi$ is of codimension at least two in $V_t$. By the semi-continuity it holds for a \emph{general} fiber $V_t$.   Since $ \overline{X}\backslash \bar{a}^{-1}(\Xi)\subset X$,  we conclude that for a  general fiber $X_t$ of $X\to J(X)$, the morphism $a|_{X_t}:X_t\to V_t$ 
	 is proper in codimension one.  The theorem is proved. 
\end{proof}

\begin{lem}\label{lem:finite cover}
	Let $A$ be a semiabelian variety and let $B$ be a semiabelian subvariety of $A$. Let  $C\to B$ be an isogeny of semi-abelian varieties.  Then there is an isogeny $A'\to A$ such that   $C\times_AA'$ is a disjoint union of $C$ and the natural morphism $C\to A'$ is injective. 
\end{lem} 

\begin{proof} 
	We write $A=\bC^n/\Gamma$, where $\Gamma$ is a lattice in $\bC^n$, and let $i:\pi_1(A)\to\Gamma$ be the natural isomorphism. There is a $\bC$-vector subspace $V\subset\bC^n$ such that $B=V/(V\cap\Gamma)$. In particular, $i({\rm Im}[\pi_1(B)\to\pi_1(A)])=\Gamma\cap V$. 
	Since $C\to B$ is an isogeny, the subgroup
	$$
	H:=i({\rm Im}[\pi_1(C)\to\pi_1(A)])
	$$
	has finite index in $\Gamma\cap V$.
	
Since the subgroup $\Gamma\cap V$ is primitive in the free abelian group $\Gamma$, we may  choose a subgroup $\Gamma_{\rm comp}$ such that $\Gamma=(\Gamma\cap V)\oplus\Gamma_{\rm comp}$. Set $\Gamma':=H\oplus\Gamma_{\rm comp}$. Then $\Gamma'$ has finite index in $\Gamma$, and
	$$
	\Gamma'\cap V=H,\qquad \Gamma'+(\Gamma\cap V)=\Gamma.
	$$
	Let $A':=\bC^n/\Gamma'$. The first equality implies that the morphism $C\to A$ lifts to an injective morphism $C\to A'$. The second equality implies that the inverse image of $B$ in $A'$ is connected, and it is identified with $C$. Consequently, $C\times_A A'=C\times_B C$. 
	Since $C\to B$ is an isogeny, we have
\[
C\times_B C \simeq  C\times \ker(C\to B).
\]
Hence $C\times_B C$ is a disjoint union of copies of $C$. The lemma is proved.
\end{proof} 
We recall a result proven in \cite{CDY22}, that is based on \cref{main2}. 
 \begin{lem}[\protecting{\cite[Lemma 7.1]{CDY22}}]\label{lem:202305101} 
	Let $X$ be a quasi-projective normal variety and let $\varrho:\pi_1(X)\to {\rm GL}_N(\bC)$ be a reductive and big representation.  
	Then there exist 
	\begin{itemize}
		\item a semi-abelian variety $A$,
		\item a smooth quasi projective variety $Y$ satisfying $\Spalg(Y)\subsetneqq Y$, 
		\item a birational modification $\widehat{X}'\to \widehat{X}$ of a finite \'etale cover $\widehat{X}\to X$,
		\item a   morphism $g:\widehat{X}'\to A\times Y$
	\end{itemize}
	such that $\dim g(\widehat{X}')=\dim \widehat{X}'$.  Moreover $p:\widehat{X}'\to Y$ is dominant, where $p$ is the composite of $g$ and the second projection $A\times Y\to Y$.
	\qed
\end{lem}

Now we prove \cref{thm:20230510} (i), (ii), (iii), which we restate as follows.

\begin{thm}[= \cref{thm:20230510} (i)-(iii)]\label{thm:structure}
	Let $X$ be a smooth quasi-projective  variety and let $\varrho:\pi_1(X)\to {\rm GL}_N(\bC)$ be a reductive and big representation.   Then 
	\begin{thmlist}
		\item \label{item:LKD} the logarithmic Kodaira dimension satisfies $\bar{\kappa}(X)\geq 0$.
Moreover, if $ \bar{\kappa}(X)= 0$, then $\pi_1(X)$ is virtually abelian.
		\item \label{item:Mori}There is a proper Zariski closed subset $\Xi$ of $X$ such that each non-constant morphism $\bA^1\to X$ has image in $\Xi$.
		\item \label{item:local trivial}After replacing $X$ by a finite \'etale cover and a birational modification, there are a    semiabelian variety $A$, a   smooth quasi-projective variety $V$  and a birational morphism $a:X\to V$  such that the following commutative diagram holds:
		\begin{equation*} 
			\begin{tikzcd}
				X \arrow[rr,    "a"] \arrow[dr, "j"] & & V \arrow[ld, "h"]\\
				& J(X)&
			\end{tikzcd}
		\end{equation*}
		where $j$ is the logarithmic Iitaka fibration of $X$ and $h:V\to J(X)$ is a  locally analytically trivial fibration with fibers isomorphic to $A$.  Moreover, for a   general fiber $F$ of $j$, $a|_{F}:F\to A$ is proper in codimension one.  
	\end{thmlist}   
\end{thm}    
\begin{proof}
Let us first show that  to prove the theorem, we are free to replace $X$ by a birational modification  and by a finite \'etale cover. For the Item (i), we first note that  the logarithmic Kodaira dimension is invariant under finite \'etale covers. 
For the fundamental group assertion, one uses surjectivity from the fundamental group of a resolution to that of the normal variety, followed by passage
from a finite-index subgroup. For Item (ii),  an $\bA^1$-curve lifts through a finite étale 
cover, and a curve not contained in the non-isomorphism locus lifts through the
proper birational modification.  Adding the relevant proper closed loci gives the
desired exceptional set downstairs.   

	We apply \cref{lem:202305101}.
Then by replacing $X$ with a finite \'etale cover and a birational modification, we obtain a smooth quasi-projective variety $Y$ (might be zero-dimensional), a semiabelian variety $\cA$, and a morphism $g:X\to \cA \times Y$ that satisfy the following properties:
	\begin{itemize}
		\item  $\dim X=\dim g(X)$.
		\item Let $p:X\to Y$ be  the composition of $g$ with the projection map $\cA\times Y\to Y$. 
		Then  $p$   is dominant.
		\item $\Spalg(Y)\subsetneqq Y$. In particular, $Y$ is of log general type if $\dim Y>0$.
	\end{itemize}

	Let us  prove \cref{item:LKD}.
	Let $\alpha:X\to\cA$ be the composite of $g:X\to \cA\times Y$ and the first projection $\cA\times Y\to \cA$.
	Let  $Z$ be a general fiber of $p$. 
	Then $\alpha|_Z:Z\to \cA$ satisfies $\dim Z=\dim \alpha(Z)$. 
	It follows that $\bar{\kappa}(Z)\geq 0$. By\cite[Theorem 1.9]{Fuj17} we obtain $\bar{\kappa}(X)\geq \bar{\kappa}(Y)+\bar{\kappa}(Z)$.  
Hence $\bar{\kappa}(X)\geq 0$.

Suppose $\bar{\kappa}(X)= 0$.
Then $\bar{\kappa}(Y)\leq0$. 
By $\Spalg(Y)\subsetneqq Y$, we conclude that $\dim Y=0$.
Hence $\dim\alpha(X)=\dim X$.
By \cref{lem:abelian pi}, $\pi_1(X)$ is abelian.
The first claim is proved.

	\medspace
	
Let us  prove \cref{item:Mori}.  
Let $E\subsetneqq X$ be a proper Zariski closed set such that $g|_{X\backslash E}:X\backslash E\to \cA\times Y$ is quasi-finite.
Set $\Xi=E\cup p^{-1}(\Spalg(Y))$.
Then $\Xi\subsetneqq X$.

We shall show that every non-constant algebraic morphism $\bA^1\to X$ has its image in $\Xi$.
Indeed, suppose $f:\bA^1\to X$ satisfies $f(\bA^1)\not\subset \Xi$.
Since $\cA$ does not contain $\bA^1$-curve, the composite $\alpha\circ f:\bA^1\to \cA$ is constant.
By $p\circ f(\bA^1)\not\subset \Spalg(Y)$, $p\circ f:\bA^1\to Y$ is constant.
Hence $g\circ f:\bA^1\to \cA\times Y$ is constant.
By $f(\bA^1)\not\subset E$, $f$ is constant.
Thus we have proved that every non-constant algebraic morphism $\bA^1\to X$ has its image in $\Xi$. 

Finally \cref{item:local trivial} follows from \cref{thm:202305101}.
\end{proof}

\begin{rem}
If we assume the logarithmic abundance conjecture: a smooth quasi-projective variety is $\bA^1$-uniruled if and only if $\bar{\kappa}(X)=-\infty$, then it predicts that   $\bar{\kappa}(X)\geq 0$ if there is a big representation $\varrho:\pi_1(X)\to {\rm GL}_N(\bC)$, which is slightly stronger than the first claim in \cref{thm:structure}. Indeed, since $\varrho$ is big, $X$ is not $\bA^1$-uniruled and thus $\bar{\kappa}(X)\geq 0$ by this conjecture.  
\end{rem}

\subsection[characterization of  semi-abelian variety]{A characterization of varieties birational to semi-abelian variety} 
In \cite{Yam10}, the third author established the following theorem: Let $X$ be a smooth projective variety equipped with a big representation $\varrho:\pi_1(X)\to {\rm GL}_N(\bC)$. If $X$ admits a Zariski dense entire curve, then after replacing $X$ by a  finite \'etale cover, its Albanese morphism $\alpha_X:X\to \cA_X$ is birational.

In \cref{thm:20230510}.(iv) we state a similar result  for smooth quasi-projective varieties $X$, provided that $\varrho$ is a reductive representation.   
\begin{proposition}[=\cref{thm:20230510}.(iv)]\label{thm:char}
	Let $Y$ be an $h$-special or   special smooth quasi-projective variety.    If $\varrho:\pi_1(Y)\to {\rm GL}_N(\bC)$ is  a  reductive and  big representation, then there is a finite \'etale cover $X$ of $Y$ such that   the quasi-Albanese morphism $\alpha:X\to \cA$ is  birational  and the induced morphism $\alpha_*:\pi_1(X)\to\pi_1(\cA)$ is an isomorphism.
In particular, $\pi_1(Y)$ is virtually abelian.
\end{proposition}

\begin{proof}
 By \cref{prop:factor} $\alpha$ is dominant with connected general fibers.	Since $\varrho$ is reductive, by \cref{main5},  there is a finite \'etale cover $X$ of $Y$   such that  $G:=\varrho(\pi_1(X))$ is  abelian and torsion free. It follows that $\varrho$ factors through  $H_1(X,\bZ)/{\rm torsion}$. 
	Since $\alpha_*:H_1(X,\bZ)/{\rm torsion}\to H_1(\cA,\bZ)$ is an isomorphism,  $\varrho$ further factors through 	$H_1(\cA,\bZ)$.
	\begin{equation*}
		\begin{tikzcd}
			&	\pi_1(X) \arrow[d]\arrow[r, "\rho"] \arrow[ldd, bend right=30]& G\\
			&	H_1(X,\bZ)/{\rm torsion} \arrow[ru]\arrow[d, "\alpha_*"']&\\
			\pi_1(\cA)\arrow[r, "="]	&	H_1(\cA,\bZ)\arrow[ruu, bend right=30,, "\beta"']&
		\end{tikzcd}
	\end{equation*}
	From the above diagram for every fiber $F$ of $\alpha$,   $\varrho({\rm Im}[\pi_1(F)\to \pi_1(X)])$ is trivial.   Since $\varrho$ is big,  the general fiber of $\alpha$ is thus a point. Hence $\alpha$ is birational.  Since $\alpha:X\to \cA$ is $\pi_1$-exact by \cref{pro:202210131}, it follows that  $\alpha_*:\pi_1(X)\to \pi_1(\cA)$ is an isomorphism. 
\end{proof}
\begin{rem}\label{rem:sharp abelian} 
We would like to point out that \cref{thm:char} is a sharp result:
	  \begin{itemize}
	 	\item In contrast to the projective case which was proven in \cite{Yam10}, we require additionally $\varrho$ to be reductive for the result to hold.
	 	\item Unlike the situation described in \cref{lem:abelian pi0},  we cannot expect that the quasi-Albanese morphism $\alpha:X\to \cA$ is proper in codimension one.
	 \end{itemize}
 In the next subsection, we will provide examples to illustrate the above points.
\end{rem}

 \subsection{Remarks on Proposition~\ref{thm:char}} 
In this subsection, we will provide examples to demonstrate the facts in \cref{rem:sharp abelian}.
\begin{lem}\label{lem:exlogCY}
	Let $X$ be the 	smooth quasi-projective surface constructed  in \Cref{example}. Then 
	\begin{itemize}
		\item the variety $X$ is a log Calabi-Yau variety, i.e. there is a smooth projective compactification $\overline{X}$ of  $X$ with $D:=\overline{X}\backslash X$ a simple normal crossing divisor such that $K_{\overline{X}}+D$ is trivial. 
		In particular, $\bar{\kappa}(X)=0$.
		\item The  fundamental group  $\pi_1(X)$ is linear and large, i.e. for any positive-dimensional closed  subvariety $Y\subset X$, the image ${\rm Im}[\pi_1(Y^{\rm norm})\to \pi_1(X)]$ is infinite.
		\item The quasi-Albanese morphism of $X$ is a fibration over an elliptic curve $B$ with fibers $\bC^*$.
		\item For any finite \'etale cover $\nu:\widehat{X}\to X$, its quasi-Albanese morphism $\alpha_{\widehat{X}}:\widehat{X}\to \cA_{\widehat{X}}$ is a surjective morphism to an elliptic curve $\cA_{\widehat{X}}$.
\end{itemize}
\end{lem}
\begin{proof}
The first assertion follows from the construction in
\cref{claim:hspecial}. Indeed, in the notation used there,
$\overline X$ is a smooth projective compactification of $X$,
its boundary is the simple normal crossing divisor $Z+D$, and
\eqref{eqn:20260925} gives that $K_{\overline X}+Z+D\sim 0$.

To prove the remaining assertions, we use the following notation.
By the construction of $X$ in \Cref{example}, there exists a holomorphic line bundle $L$ over an elliptic curve $B$  such that $X=L\backslash Z$ where $Z$ is the zero section of $L$. 
    Now by the Gysin sequence, we have 
    \begin{align}\label{eq:Gysin}
    	0\to H^1(B,\bZ)\stackrel{\pi^*}{\to} H^1(X,\bZ)\to H^0(B,\bZ)\stackrel{\cdot c_1(L)}{\to} H^2(B,\bZ)\to H^2(X,\bZ)\to H^1(B,\bZ)\to\cdots
    \end{align}   
    where $\pi:X\to B$ is the projection map.    By the functoriality of the quasi-Albanese morphism, we have the following diagram
\begin{equation*}
	\begin{tikzcd}
		X \arrow[r,"\pi"]\arrow[d,"\alpha_X"] & B\arrow[d,"\alpha_B","\simeq"']\\
		\cA_X\arrow[r,"\simeq"] & \cA_{B}
	\end{tikzcd}
\end{equation*}
where $\alpha_X$ and $\alpha_B$ are (quasi-)Albanese morphisms of $X$ and $B$ respectively.     By \Cref{re:20230428}, we know that $c_1(L)\neq 0$.
It follows from \eqref{eq:Gysin} that $\pi^*:H^1(B,\bZ)\to H^1(X,\bZ)$ is an isomorphism, and thus $\cA_X\to \cA_B$ is an isomorphism.   Since $B$ is an elliptic curve, $\alpha_B:B\to \cA_B$ is also an isomorphism. It proves that $\pi$ coincides with the quasi-Albanese morphism   $\alpha_X$. 

We will prove that $\pi_1(X)$ is large. Since $\pi_1(X)$ is infinite, it suffices to check all irreducible curves $Y$ of $X$.    If $Y$ is a fiber of $\pi$, then $Y\simeq \bC^*$ and the claim follows from the exact sequence
$$
0=\pi_2(B)\to \pi_1(\bC^*)\to \pi_1(X)\to \pi_1(B).
$$   
If $Y$ is not a fiber of $\pi$, then $\pi|_{Y}:Y\to B$ is a  dominant morphism, and thus   ${\rm Im}[\pi_1(Y^{\rm norm})\to \pi_1(B)]$ is a finite index subgroup of $\pi_1(B)$ which is thus infinite. It follows that   ${\rm Im}[\pi_1(Y^{\rm norm})\to \pi_1(X)]$ is infinite. 

Let us prove the last assertion. 
\begin{claim}
	 The natural morphism  $\nu^*: H^1(X, \mathbb{C}) \to H^1(\widehat{X}, \mathbb{C})$ is an isomorphism. 
\end{claim} \begin{proof}
Set $G:=\pi_1(X)$ and $H:=\pi_1(\widehat X)$. Then $H$ is a finite-index subgroup of $G$. Recall that $G$ is a Heisenberg group. In particular, $Z(G)$ has rank one, $G/Z(G)\simeq\mathbb Z^2$, and $[G,G]$ has finite index in $Z(G)$. Then  $H\cap Z(G)$ has rank one, $[H,H]$ has finite
index in this intersection, and the image of $H$ in $G/Z(G)\simeq \bZ^2$ has finite index.
Thus
$H^{ab}\otimes \bQ=\bQ^2$, which proves that
 $
b_1(\widehat{X})=2. 
$  Since $b_1(X)=2$ and the pullback
$$
\nu^*:H^1(X,\mathbb C)\longrightarrow H^1(\widehat X,\mathbb C)
$$
is injective, it is an isomorphism.
\end{proof}From this, we deduce that $\dim_\bC H^1(\cA_{\widehat{X}},\bC)=\dim_\bC H^1(\widehat{X},\bC)=\dim_\bC H^1({X},\bC)=2$.  Therefore, $ \cA_{\widehat{X}}$ is   either an elliptic curve or $(\bC^*)^2$.  However, it is worth noting that $\nu$ induces a non-constant morphism $\cA_{\widehat{X}} \to \cA_X=B$. This implies that $\cA_{\widehat{X}}$ cannot be $(\mathbb{C}^*)^2$. To see this, consider the algebraic morphism from $\mathbb{C}^*$ to an elliptic curve $B$, which can be extended to $\mathbb{P}^1 = \mathbb{C}^* \cup \{0,\infty\}$. This extension must be constant, ruling out the possibility of $\cA_{\widehat{X}}$ being $(\mathbb{C}^*)^2$.

Let us prove that quasi-Albanese variety $\alpha_{\widehat X}$ of $\widehat X$ is surjective. As shown above, the quasi-Albanese variety of $\widehat X$ is an elliptic curve, which we denote by $E$. Let $\alpha_{\widehat X}:\widehat X\to E$ be the quasi-Albanese morphism. By \cref{lem:202304061}, $\widehat X$ is $h$-special. If $\alpha_{\widehat X}$ were not surjective, it would omit a point $q\in E$ and hence induce a dominant morphism $\widehat X\to E\setminus\{q\}$. This contradicts \cref{lem:202304063,lem:202304065}. Therefore, $\alpha_{\widehat X}$ is surjective.
\end{proof}
The above lemma shows that \cref{thm:char} does not hold if $\varrho$ is not reductive.

\begin{lem} \label{lem:20230507}
	Let $X$ be the quasi-projective surface constructed in \Cref{example:20221105}, which is special and $h$-special.   Then
	\begin{itemize}
	 \item for the quasi-Albanese morphism $\alpha:X\to \cA_X$,   $\alpha_*:\pi_1(X)\to \pi_1(\cA_X)$ is an isomorphism.  
\item there is a faithful and big linear representation for $\pi_1(X)$.
	\item For any finite \'etale cover $\nu:\widehat{X}\to X$, its quasi-Albanese morphism $\alpha_{\widehat{X}}:\widehat{X}\to \cA_{\widehat{X}}$ is birational but not proper in codimension one. 
	\end{itemize} 
\end{lem}

\begin{proof}
	We use the notation of \Cref{example:20221105}. Let $q:X\to A=C_1\times C_2$ be the birational morphism given by the construction. Since its general fiber is a point,  it follows from \cref{pro:202210131} that 
	$$
	q_*:\pi_1(X)\xrightarrow{\sim}\pi_1(A)\simeq\mathbb Z^4.
	$$
	By the universal property of the quasi-Albanese morphism, we have a factorization
	$$
	\begin{tikzcd}
		X \arrow[r, "\alpha"] \arrow[rr, bend left=25, "q"]&
		\cA_X \arrow[r, "\theta"]& A.
	\end{tikzcd}
	$$
	where $\alpha:X\to \cA_X$ is the quasi-Albanese map. 
	By \cref{prop:factor}, $\alpha$ is dominant. Hence $\dim\cA_X\leq 2$. On the other hand, the dominance of $q$ implies that $\dim\cA_X\geq 2$. Therefore, $\dim\cA_X=2$, the morphism $\theta$ is an isogeny, and $\alpha$ is generically finite. Since $q$ is birational,   $\theta$ is an isomorphism and, under the identification $\cA_X\simeq A$, the morphism $\alpha$ coincides with $q$. In particular, $\alpha_*=q_*$ is an isomorphism.
	
	We next prove the second assertion. By the first assertion, $\pi_1(X)\simeq\mathbb Z^4$. Let $\varrho:\mathbb Z^4=\pi_1(X)\to{\rm GL}_N(\mathbb C)$ be a faithful diagonal representation, which is reductive. Note that $ \pi_1(A)$ has large fundamental group,  it follows that $\varrho$ is a big and reductive representation. 
	 
	 We observe that $g(X)\subset (C_1-{0})\times C_2\cup {(0,0)}$. As a result, it is not proper in codimension one.
 
 Let us prove the last assertion. Since $\pi_1(\widehat{X})$ is a finite index subgroup of $\pi_1(\cA_X)$ and $\alpha_*:\pi_1(X)\to \pi_1(\cA_X)$ is an isomorphism, we can consider a finite étale cover $ \widehat{\cA} \to \cA_X$ associated with the finite index subgroup  $\alpha_*(\pi_1(\widehat{X}))$   of $\pi_1(\cA_X)$. Then  $\widehat{\cA}$ is also an abelian variety and there exists  a morphism $f:\widehat{X}\to \widehat{\cA}$ satisfying the following commutative diagram 
 \begin{equation*}
 	\begin{tikzcd}
 		\widehat{X} \arrow[r,"\pi"]\arrow[d,"f"] & X\arrow[d,"\alpha"]\\
 		\widehat{\cA}\arrow[r] & \cA_{X}
 	\end{tikzcd}
 \end{equation*}
 As $\alpha$ is birational, we have $\widehat{X}=X\times_{\cA_X}\widehat{\cA}$ which implies that $f$ is birational but not proper in codimension one as $\alpha$ is not proper in codimension one.   Furthermore, since $f_*: \pi_1(\widehat{X}) \to \pi_1(\widehat{\cA})$ is an isomorphism, by the same arguments as above we conclude that   $f$ coincides with the quasi-Albanese morphism $\alpha_{\widehat{X}}:\widehat{X}\to \cA_{\widehat{X}}$.   
This completes the proof of the last claim.
\end{proof}
The above lemma shows that in \cref{thm:char} we cannot expect that the quasi-Albanese morphism is   proper in codimension one.

% \bibliography{biblio2}
%  \bibliographystyle{smfalpha-url}

\providecommand{\bysame}{\leavevmode ---\ }
\providecommand{\og}{``}
\providecommand{\fg}{''}
\providecommand{\smfandname}{\&}
\providecommand{\smfedsname}{\'eds.}
\providecommand{\smfedname}{\'ed.}
\providecommand{\smfmastersthesisname}{M\'emoire}
\providecommand{\smfphdthesisname}{Th\`ese}

\end{document}